\documentclass[11pt, twoside]{article}

\usepackage{amsmath}
\usepackage{amssymb}
\usepackage{titletoc}
\usepackage{mathrsfs}
\usepackage{amsthm}

\usepackage{indentfirst}
\usepackage{color}
\usepackage{txfonts}

\allowdisplaybreaks

\def\bint{{\ifinner\rlap{\bf\kern.30em--}
\int\else\rlap{\bf\kern.35em--}\int\fi}\ignorespaces}

\def\sbint{{\ifinner\rlap{\bf\kern.32em--}
\hspace{0.078cm}\int\else\rlap{\bf\kern.45em--}\int\fi}\ignorespaces}

\newtheorem{theorem}{Theorem}[section]
\newtheorem{lemma}[theorem]{Lemma}
\newtheorem{corollary}[theorem]{Corollary}
\newtheorem{proposition}[theorem]{Proposition}

\theoremstyle{definition}
\newtheorem{remark}[theorem]{Remark}
\newtheorem{definition}[theorem]{Definition}

\numberwithin{equation}{section}

\begin{document}

\title{\bf\Large Wavelet Characterization of 
Inhomogeneous Lipschitz Spaces on 
Spaces of Homogeneous Type and Its Applications
\footnotetext{\hspace{-0.35cm} 2020 {\it Mathematics
Subject Classification}. Primary 46E36;
Secondary 46E35, 46E39, 42B25, 30L99.\endgraf
{\it Key words and phrases.} space of homogeneous type, 
inhomogeneous Lipschitz space, wavelet, Ahlfors regular condition.\endgraf
This project is partially supported by the
National Natural Science Foundation of China (Grant Nos. 12401115), 
the Open Project Program of Key
Laboratory of Mathematics and Complex System
of Beijing Normal University (Grant No. K202304), 
the Hebei Natural Science Foundation (Grant No. A2024201004), 
and the Excellent Youth Research Innovation 
Team of Hebei University (QNTD202414).}}
\author{Fan Wang\footnote{Corresponding author,
E-mail: \texttt{fanwang@hbu.edu.cn}/{\color {red} 
April 21, 2025}/newest version.} }
\date{}
\maketitle

\vspace{-0.8cm}

\begin{center}
\begin{minipage}{13cm}
{\small {\bf Abstract}\quad
In this article,
the author establishes  a wavelet 
characterization of inhomogeneous Lipschitz 
space $\mathrm{lip}_{\theta}(\mathcal{X})$ via 
Carlson sequence, where $\mathcal{X}$ is a space of homogeneous
type introduced by R. R. Coifman and G. Weiss. As applications, characterizations  
of several geometric conditions on $\mathcal{X}$, 
involving the upper bound, the lower bound, and the Ahlfors regular 
condition, are obtained.}
\end{minipage}
\end{center}

\vspace{0.2cm}



\section{Introduction}\label{intro}

As a very fundamental function space, 
(inhomogeneous) Lipschitz spaces 
permeate both pure and applied disciplines. 
Their significance extends ubiquitously across diverse mathematical 
domains, such as ordinary and partial 
differential equations, measure-theoretic analysis, and nonlinear 
functional analysis, as well as geometric-topological contexts 
including metric geometry, fractal theory, and topological dynamics. 
Beyond theoretical mathematics, these functions demonstrate remarkable 
versatility in computational science, finding essential applications 
in image processing algorithms, search engine optimization architectures, 
and stability analysis of machine learning models.

During the 1970s, Coifman and Weiss \cite{cw71,cw77} introduced 
the groundbreaking concept of spaces 
$(\mathcal{X},d,\mu)$ of homogeneous 
type (see
Definition \ref{d-shy} below), a framework extending classical Euclidean analysis to general metric measure spaces. 
This innovation catalyzed profound investigations into Lipschitz spaces over such structures.
In \cite{ms79}, Mac\'{i}as and Segovia made  contributions by elucidating the geometric structure of $\mathcal{X}$ and unifying several definitions of Lipschitz functions on these spaces. In 2020, Zheng et al. \cite{zlt} established the Littlewood--Paley 
characterization of Lipschitz spaces on Ahlfors regular spaces. Later, Li and Zheng \cite{lz} obtained the 
boundedness of Caldr\'on--Zygmund operators on Lipschitz spaces.  
Motivated by the advent of wavelets system constructed in \cite{ah13}, 
Liu et al. \cite{lyy} developped an wavelets characterization of homogeneous Lipschitz spaces.  
On the other hand, 
He et al.\ \cite[Definition 2.7]{hlyy} introduced a new kind 
of \emph{approximations of the
identity with exponential decay}, a pivotal tool to establish 
(in)homogeneous continuous/discrete Calder\'on reproducing
formulae on $\mathcal{X}$. 
Building on this foundation, He et al.  \cite{hhllyy}
obtained a complete real-variable theory of atomic 
Hardy spaces on $\mathcal{X}$. 
Recently, based on the concept of inhomogeneous approximation
of the identity with exponential decay, He et al. \cite{hyy19} established 
several characterizations of local Hardy space $h^p(\mathcal{X})$ and 
showed that the dual of $h^p(\mathcal{X})$ is the inhomogeneous 
Lipschitz space $\mathrm{lip}_{1/p-1}(\mathcal{X})$. 
We refer the reader to \cite{bdl18,bdl20,lj10,wyy,yhyy} for more recent progress on the topic 
of (local) Hardy spaces and their duals on spaces of homogeneous type.

Theoretically, an important significance of Lipschitz spaces lies 
in their role as the dual of Hardy-type spaces. 
Thus, the products of functions in Hardy spaces and 
Lipschitz spaces have also garnered significant research interest.
Inspired by the progress on geometric function theory (see, 
for instance, \cite{aikm}) and the nonlinear elasticity (see, 
for instance, \cite{b76,m88}), Bonami et al. \cite{bijz} pioneered the 
investigation into bilinear decompositions involving products of functions in Hardy spaces and 
Lipschitz spaces. Subsequent developments by Bonami and Feuto \cite{bf10,f09} established the linear 
decomposition of product of functions in $H^p(\mathbb{R}^n)$
and its dual space. Concurrently, Li and Peng \cite{lp09} obtained a  
linear decomposition of product of functions in $H_L^1(\mathbb{R}^n)$
and its dual space ${\rm BMO}_L(\mathbb{R}^n)$, 
where $L:=-\Delta+V$ is a Schr\"odinger operator; see also Ky 
\cite{ky14} for a bilinear version. In the context of local Hardy space 
$h^p(\mathbb{R}^n)$, Cao et al. \cite{cky18} 
established a bilinear decomposition 
of products for functions in $h^p(\mathbb{R}^n)$ and its dual spaces with $p\in(0,1]$, 
which was further refined by Yang et al. \cite{yyz21}. 
In \cite{yyz21}, Yang et al. obtained alternative bilinear 
decomposition of products for functions in 
$h^p(\mathbb{R}^n)$ and its dual spaces with $p\in(0,1)$, which was 
shown to be sharp in the dual sense. Moreover, using this bilinear 
decomposition, Yang et al. \cite{yyz21} obtained some div-curl estimates. 
These results of bilinear decomposition also 
play key roles in the estimates of weak Jacobians 
(see, for instance, \cite{bfg,bgk}) and commutators 
(see, for instance, 
\cite{ky13,lcfy17}). These works further inspire many new ideas
in the research of nonlinear partial differential equations; 
see, for instance, \cite{bijz,is01,lr02} and their references therein
 for more details.
Recent advances in (bi)linear decomposition theory for (local) Hardy space products and their duals continue to emerge,
as documented in \cite{bcklyy,bjxyz,fyl,lcfy18,lyz24}, 
highlighting the enduring vitality of this research direction.

Notice that wavelet characterization for (inhomogeneous) Lipschitz 
spaces plays a key role when analyzing products of functions 
in (local) Hardy spaces and their duals. 
This naturally raises an question: Can inhomogeneous Lipschitz 
spaces on general spaces of homogeneous 
type admit analogous wavelet characterizations? 
The main target of this paper is to give an affirmative answer. Precisely, we  
develop a wavelet characterization of the inhomogeneous Lipschitz space $\mathrm{lip}_{\theta}(\mathcal{X})$ via 
Carlson sequence. 
\begin{theorem}\label{thm-lip-cs}
Let $\omega$ be as in \eqref{eq-doub}, $\eta \in (0,1]$ 
be as in Lemma \ref{l-wave1},  
and $\theta\in (0,\eta/\omega)$. Then, 
for any $f\in L^2_{\mathcal{B}}(\mathcal{X})$, 
the following statements are equivalent:
\begin{enumerate}
\item[{\rm(i)}] $f\in\mathrm{lip}_\theta(\mathcal{X})$;
\item[{\rm(ii)}] $$
f=\sum_{\alpha\in\mathcal{A}_0} 
\left\langle f, \phi_\alpha^0\right\rangle
\phi_\alpha^0+\sum_{k=0}^\infty\sum_{\beta\in\mathcal{G}_k} 
\left\langle f, 
\psi_\beta^{k+1}\right\rangle\psi_\beta^{k+1}
$$
in $L_{\mathcal{B}}^2(\mathcal{X})$ and 
\begin{align*}
\|f\|_\ast=&{}\sup_{Q\in \mathcal{D}_0}
\left\{\frac{1}{[\mu(Q)]^{1+2\theta}}\left[
\sum_{\{\alpha\in\mathcal{A}_0:Q_\alpha^0\subset Q\}}\left|
\left\langle f, \phi_\alpha^0\right\rangle\right|^2\right.\right.\\
&\quad\left.\left.+
\sum_{k=0}^\infty\sum_{\{\beta\in\mathcal{G}_k:
Q_\beta^{k+1}\subset Q\}}\left|\left\langle f, 
\psi_\beta^{k+1}\right\rangle\right|^2\right]\right\}^{\frac{1}{2}}
<\infty. 
\end{align*}
\end{enumerate}
Moreover, there exists a constant $C\in[1,\infty)$, such that 
$$
C^{-1}\|f\|_\ast\leq \|f\|_{\mathrm{lip}_\theta(\mathcal{X})}
\leq C\|f\|_\ast.
$$
\end{theorem}
This 
result crucially eliminates dependencies on  the reverse doubling
condition of the measure and the metric condition of the quasi-metric under consideration. Moreover, using this  
characterization, we discuss several geometric conditions on $\mathcal{X}$, involving the upper bound, the lower bound, and the Ahlfors regular 
condition, and obtain some equivalence characterizations.  

The organization of the remainder of this article is as follows.

In Section \ref{s-gc},  we first recall some basic preliminaries 
on spaces of homogeneous type, inhomogeneous 
Lipschitz spaces, dyadic cube system established in \cite{hk}, 
spaces of test functions, and spaces of distributions. 
We show that all test functions are 
pointwise multipliers on the inhomogeneous Lipschitz spaces; 
see Proposition \ref{p-pm} below. 

Section \ref{proof} is devoted to proving Theorem \ref{thm-lip-cs}. 
To this end, we first recall 
the wavelets system obtained in \cite{ah13}. 
Using theses wavelets, we establish an 
equivalence characterization of  imhomogenous Lipschitz spaces  
via  Carleson sequences. 
Differently from the proof 
of \cite[Theorem 2.8]{cky18} on Euclidean spaces $\mathbb{R}^n$, 
the relation 
$\mathrm{lip}_\theta(\mathcal{X})=
F_{\infty,\infty}^{\omega\theta}(\mathcal{X})
=C^{\omega\theta}(\mathcal{X})$
may not holds true in the general setting of spaces of 
homogeneous type; see Corollary \ref{thm-ags} below. 
To overcome this, we fully use the exponential decay of 
the wavelets to obtain that the integrals of 
the products of functions in  $\mathrm{lip}_\theta(\mathcal{X})$ and  
wavelets also have enough decay; see Proposition \ref{lip-wave}
below. 

In Section \ref{app}, we give some applications. 
As corollaries of Theorem \ref{thm-lip-cs}, we develop three equivalent 
characterizations of geometric conditions on 
$\mathcal{X}$, involving the upper bound, the lower bound, 
and the Ahlfors regular 
condition. Corollary \ref{thm-ags} extends  
results in \cite[Theorem 3.2]{wyy} into 
the inhomogeneous version. 

Finally, let us make some conventions on notation.
Throughout this article, $A_0$ is used to denote
the positive constant appearing in \eqref{tri-in}, $\omega$ is used to
to denote the \emph{upper dimension} in  \eqref{eq-doub},
and $\eta$ is used to denote
the smoothness index of wavelets in Lemma \ref{l-wave1}.
Moreover, $\delta$ is a small positive number
coming from the construction of the
dyadic cubes on $\mathcal{X}$ (see Lemma \ref{2-cube} below).
We use $C$ to denote a positive constant which is
independent of the main parameters involved, but may vary
from line to line.  The symbol $C_{(\alpha,\beta,\dots)}$
 denotes a positive constant depending on the indicated
parameters $\alpha,\ \beta,\ \dots$.
The symbol $A\lesssim B$ means that $A\leq CB$ for some
positive constant $C$, while $A\sim B$ means $A\lesssim B\lesssim A$.
If $f\le Cg$ and $g=h$
or $g\le h$, we then write $f\lesssim g=h$ or $f\lesssim g\le h$.
For any set $E\subset \mathcal{X}$, 
$\mathbf{1}_E$ means the characteristic function of $E$. 
For any set $F$,
$\#F$ denotes its \emph{cardinality}.

\section{Inhomogeneous Lipschitz Spaces on 
Spaces of Homogeneous Type}\label{s-gc}

In this section, we mainly investigate the inhomogeneous Lipschitz 
spaces $\mathrm{lip}_\theta(\mathcal{X})$ on spaces of 
homogeneous type, including  the relation between 
$\mathrm{lip}_\theta(\mathcal{X})$ and distributions spaces. Moreover, 
we establish an equivalent characterization 
of $\mathrm{lip}_\theta(\mathcal{X})$ via  Carleson sequences. 
Let us first recall the concept of space of homogeneous type in the sense of 
Coifman and Weiss \cite{cw71,cw77}.

\begin{definition}\label{d-shy}
Let $\mathcal{X}$ be a non-empty set, $d$ a non-negative function
defined on $\mathcal{X}\times\mathcal{X}$, and $\mu$ a measure 
on $\mathcal{X}$. $(\mathcal{X},d,\mu)$ is called
a \emph{space of homogeneous type} provided that $d$ and $\mu$
satisfy the following conditions:
\begin{enumerate}
\item[{\rm (I)}]for any $x, y, z \in \mathcal{X}$,
\begin{enumerate}
\item[{\rm(i)}] $d(x,y)=0$ if and only if $x=y$;
\item[{\rm(ii)}] $d(x,y)=d(y,x)$;
\item[{\rm(iii)}] there exists a constant $A_0 \in [1, \infty)$,
independent of $x$, $y$, and $z$, such that
\begin{equation}\label{tri-in}
d(x,z)\leq A_0[d(x,y)+d(y,z)];
\end{equation}
\end{enumerate}
\item[{\rm (II)}] there exists a constant $C\in[1,\infty)$ such that,
for any ball $B \subset \mathcal{X}$,
\begin{equation}\label{eq-doub2}
\mu(2B)\leq C\mu(B),
\end{equation}
where, the \emph{ball} $B$, centered at $x_B \in \mathcal{X}$
with radius $r_B\in(0, \infty)$, of $\mathcal{X}$ is defined by setting
$$
B:=B(x_B,r_B):=\left\{x\in\mathcal{X}:\ d(x_B,x)<r_B\right\}
$$
and, for any $\tau\in(0,\infty)$, $\tau B:=B(x_B,\tau r_B)$.
\end{enumerate}
\end{definition}

In what follows, we always
assume that $\mu(\mathcal{X})=\infty$. 
Note that  $\rm{diam} \mathcal{X}=\infty$ implies
$\mu(\mathcal{X})=\infty$ (see, for instance,
\cite[Lemma 5.1]{ny97} and \cite[Lemma 8.1]{ah13}).
Therefore, under
the assumptions of this article, $\mu(\mathcal{X})=\infty$ if
and only if $\rm{diam} \mathcal{X}=\infty$.
For any $x\in \mathcal{X}$, we also assume that the 
balls $\{B(x,r)\}_{r\in(0,\infty)}$ form a basis of open neighborhoods
of $x$. Moreover, we also assume that $\mu$ is Borel regular, 
that is, all open sets are measurable and every set
$A\subset \mathcal{X}$ is contained in a Borel set $E$ satisfying that $\mu(A)=\mu(E)$.
For any $x\in \mathcal{X}$ and $r\in(0,\infty)$, we suppose 
that $\mu(B(x,r))\in (0,\infty)$ and $\mu(\{x\})=0$.
Let 
$$C_{(\mu)}:=\sup_{\mathrm{ball}\,B\subset \mathcal{X}}\mu(2B)/\mu(B).$$
Then it is easy to prove that  $C_{(\mu)}$ is the smallest
positive constant satisfying \eqref{eq-doub2}. Moreover,
 \eqref{eq-doub2} further implies that,
for any ball $B$ and any $\lambda\in[1,\infty)$,
\begin{equation}\label{eq-doub}
\mu(\lambda B)\leq C_{(\mu)}\lambda^\omega\mu(B),
\end{equation}
where $\omega:= \log_2 C_{(\mu)}$ is called
the \emph{upper dimension} of $\mathcal{X}$.

The following lemma includes  some useful estimates
related to the measure of balls;
see, for instance, \cite[Lemma 2.1]{hmy08} for more 
details (see also \cite[Lemma 2.4]{hlyy}).
For any $r\in(0,\infty)$ and $x,y\in\mathcal{X}$ with $x\neq y$,
let 
$$V(x,y):=\mu(B(x,d(x,y)))\ \mathrm{and}\ V_r(x):=\mu(B(x,r)).$$

\begin{lemma}\label{l-ball}
\begin{enumerate}
\item[{\rm (i)}]Let $x,y\in \mathcal{X}$ with $x\neq y$ and $r \in (0, \infty)$.
Then $V(x, y)\sim V (y, x)$ and
\begin{align*}
V_r(x) + V_r(y) + V (x,y) &\sim V_r(x) + V (x,y) \sim V_r(y) + V (x,y)\\
&\sim \mu(B(x,r + d(x,y))).
\end{align*}
Moreover, if $d(x,y)\leq r$, then $V_r(x)\sim V_r(y)$.
Here all the above positive equivalence constants are independent
of $x$, $y$, and $r$.
\item[{\rm(ii)}] Let $\gamma\in(0,\infty)$. 
There exists a positive constant $C$ such that,
for any $x_1 \in \mathcal{X}$ and $r \in (0, \infty)$,
$$
\int_{\mathcal{X}} \frac{1}{V_r(x_1)+V(x_1,y)}\left[\frac{r}{r+d(x_1,
y)}\right]^\gamma\,d\mu(y)\leq C.
$$
\item[{\rm(iii)}] There exists a positive constant $C$ such that, 
for any $x \in \mathcal{X} $ and $R \in (0, \infty)$,
$$\quad \int_{\{z\in \mathcal{X}:\ d(x,z)\geq R\}}\frac{1}{V(x,y)}
\left[\frac{R}{d(x,y)}\right]^\beta\,d\mu(y)\leq C.$$
\end{enumerate}
\end{lemma}

On the ratio of measures of two balls, we have the following lemma. 

\begin{lemma}\label{l-ball2}
Let $x,y\in\mathcal{X}$ and $r_1,r_2\in(0,\infty)$. If $r_1+d(x,y)\geq r_2$, 
then 
$$
\frac{\mu(B(x,r_1))}{\mu(B(y,r_2))}\leq A_0^\omega
\left[\frac{r_1+d(x,y)}{r_2}\right]^\omega.
$$
\end{lemma}

\begin{proof}
Let $x,y\in\mathcal{X}$ and $r_1,r_2\in(0,\infty)$.  By \eqref{tri-in}, we find that, for 
any $z\in B(x,r_1)$, 
$$
d(z,y)\leq A_0[d(z,x)+d(x,y)]\leq A_0[r_1+d(x,y)],
$$
which further implies that 
$$B(x,r_1)\subset B(y,A_0[r_1+d(x,y)])
=B\left(y,A_0\frac{r_1+d(x,y)}{r_2}r_2\right).$$
This, together with \eqref{eq-doub}, further implies that 
$$
\mu(B(x,r_1))\leq A_0^\omega
\left[\frac{r_1+d(x,y)}{r_2}\right]^\omega\mu(B(y,r_2)),
$$
which completes the proof of Lemma \ref{l-ball2}.
\end{proof}

Now, we recall the concept of inhomogeneous  Lipschitz spaces on spaces
of homogeneous type. For any $q\in[1,\infty]$, the set 
$L^q_{\mathcal{B}}(\mathcal{X})$ denotes the collection of all
measurable functions $f$ on $\mathcal{X}$ such that 
$f\mathbf{1}_B\in L^q(\mathcal{X})$ for any ball 
$B\subset \mathcal{X}$. For any 
$\{f_n\}_{n\in\mathbb{N}}\subset L^q_{\mathcal{B}}(\mathcal{X})$ 
and $f\in L^q_{\mathcal{B}}(\mathcal{X})$, 
if, for any $B\subset \mathcal{X}$, 
$$
\lim_{n\to\infty} \|f_n-f\|_{L^q(B)}=0,
$$
then we say that $\{f_n\}_{n\in\mathbb{N}}$ converges to $f$ in 
$L^q_{\mathcal{B}}(\mathcal{X})$.

\begin{definition}\label{d-lip}
For any $B:=B(x_B,r_B)\subset \mathcal{X}$, $\theta\in(0,1)$, 
and $f\in L^q_{\mathcal{B}}(\mathcal{X})$, let
$$
\mathcal{M}_\theta^B(f):=\begin{cases}
\displaystyle{\sup_{x,y\in B} \frac{|f(x)-f(y)|}{[
\mu(B)]^\theta}}& \textup{if } r_B\in(0,1],\\
{}&{}\\
\displaystyle{\frac{\|f\|_{L^\infty(B)}}{[
\mu(B)]^\theta}} & \textup{if } r_B\in(1,\infty).
\end{cases}
$$
The \emph{inhomogeneous Lipschitz spaces} $\mathrm{lip}_\theta(
\mathcal{X})$ is 
defined by setting 
$$
\mathrm{lip}_\theta(\mathcal{X}):=\left\{f\in
L^q_{\mathcal{B}}(\mathcal{X})
: \|f\|_{\mathrm{lip}_\theta(\mathcal{X})}:=\sup_{\text{ball}\ 
B\subset \mathcal{X}}
\mathcal{M}_\theta^B(f)<\infty\right\}.
$$
\end{definition}

By Definition \ref{d-lip}, we find that, 
for any $f\in \mathrm{lip}_\theta(\mathcal{X})$, 
$f$ is continuous. Moreover, we have the following properties of   
$\mathrm{lip}_\theta(\mathcal{X})$.

\begin{lemma}\label{l-point}
Let $\theta\in(0,1)$. Then there exists a 
positive constant $C$ such that, 
for any $x\in\mathcal{X}$, $r\in(1,\infty)$, and 
$f\in \mathrm{lip}_\theta(\mathcal{X})$,
\begin{equation}\label{e-p}
|f(x)|\leq C\|f\|_{\mathrm{lip}_\theta(\mathcal{X})}[V_r(x)]^\theta,
\end{equation}
and, for any $x,y\in\mathcal{X}$ with $x\neq y$, 
\begin{equation}\label{e-p2}
|f(x)-f(y)|\leq C\|f\|_{\mathrm{lip}_\theta(\mathcal{X})}[V(x,y)]^\theta.
\end{equation}
\end{lemma}

To prove Lemma \ref{l-point}, we need the following dyadic cube system
established by Hyt\"onen and Kairema in \cite[Theorem 2.2]{hk}.

\begin{lemma}\label{2-cube}
Let $c_0,C_0 \in(0,\infty)$ and $\delta\in(0, 1)$
be such that $c_0<C_0$ and $12A_0^3C_0\delta\leq c_0$.
Assume that $\mathcal{A}_k$, a set of indices for
any $k \in\mathbb{Z}$, and a set of points, $\{ x_\alpha^k:\ k \in\mathbb{Z} ,
\alpha \in \mathcal{A}_k \} \subset \mathcal{X}$, have the following
properties: for any $k \in\mathbb{Z}$,
\begin{equation}\label{dis-re}
d(x^k_\alpha, x^k_\beta)\geq c_0\delta^k\
\text{if}\ \alpha\neq\beta,\
\text{and}\ \min_{\alpha\in\mathcal{A}_k}d(x,x_\alpha^k)
<C_0\delta^k\ \text{for any}\ x\in\mathcal{X}.
\end{equation}
Then there exists a family of sets,
$\{ Q_\alpha^k : k \in\mathbb{Z} , \alpha \in \mathcal{A}_k \}$, such that
\begin{enumerate}
\item[{\rm(i)}] for any $k\in\mathbb{Z}$, 
$\{ Q_\alpha^k:\ \alpha \in \mathcal{A}_k \}$
is disjoint and $\bigcup_{\alpha\in\mathcal{A}_k}
Q_\alpha^k=\mathcal{X}$;
\item[{\rm(ii)}] if $l, k\in\mathbb{Z}$ and $l\leq k$, then,
for any $\alpha\in\mathcal{A}_l$ and $\beta\in\mathcal{A}_k$,
either $Q_\beta^k\subset Q_\alpha^l$
or $Q_\beta^k\cap Q_\alpha^l=\emptyset$;
\item[{\rm(iii)}] for any $k\in\mathbb{Z}$ and $\alpha\in\mathcal{A}_k$,
$B(x^k_\alpha, c_\#\delta^k)\subset Q_\alpha^k
\subset B(x^k_\alpha, C_\#\delta^k)$, where $c_\#:=(3A_0^2)^{-1}c_0$
and $C_\#:=\max\{2A_0C_0,1\}$.
\end{enumerate}
\end{lemma}

Points in $\{ x_\alpha^k:\ k \in\mathbb{Z} ,
\alpha \in \mathcal{A}_k \} \subset \mathcal{X}$ are 
called \emph{dyadic points}.
For any $k\in\mathbb{Z}$, let 
$$\mathcal{X}^k:=\left\{x_\alpha^k:\ \alpha\in\mathcal{A}_k\right\}.$$
By the construction of dyadic points in \cite[2.21]{hk},
we may assume that $\mathcal{X}^k$ is countable and
$\mathcal{X}^k\subset \mathcal{X}^{k+1}$ for any $k\in\mathbb{Z}$. 
For any $k\in\mathbb{Z}$, define
$$\mathcal{Y}^k:=
\mathcal{X}^{k+1}\setminus\mathcal{X}^k$$
and
$$\mathcal{G}_k:=\left\{\alpha\in\mathcal{A}_{k+1}:\ y_\alpha^k:=
x_\alpha^{k+1}\in\mathcal{Y}^k\right\}.$$
For any $x \in \mathcal{X}$, let
$$d(x,\mathcal{Y}^k):=\inf_{y\in\mathcal{Y}^k}d(x,y).$$

Next, we prove Lemma \ref{l-point}.

\begin{proof}[Proof of Lemma \ref{l-point}]
Notice that $f$ is continuous. To show \eqref{e-p}, 
it suffices to prove that \eqref{e-p} holds true for almost every 
$x\in\mathcal{X}$. For this purpose, assume that 
\begin{equation}\label{c-1}
E:=\left\{x\in\mathcal{X}: |f(x)| > C\|f\|_{\mathrm{lip}_\theta(
\mathcal{X})}[V_r(x)]^\theta\right\} \ \text{and}\  \mu(E)>0,
\end{equation}
where $C$ is determined later. 
Find $k_0\in\mathbb{Z}$ such that 
$C_\#\delta^{k_0}< r$ where $C_\#$ is the same 
as in Lemma~\ref{2-cube}(iii). 
For any $\alpha\in \mathcal{A}_{k_0}$, let $E_\alpha:=E\cap 
B(x_\alpha^{k_0},r)$.
Then, by (i) and (ii) 
of Lemma \ref{2-cube}, we have 
$E=\bigcup_{\alpha\in \mathcal{A}_{k_0}} E_\alpha.$
From this, the fact that $\mathcal{A}_{k_0}$ is countable, and 
\eqref{c-1}, we deduce that there exists $\alpha_0\in \mathcal{A}_{k_0}$
such that 
\begin{equation}\label{c-2}
\mu\left(E_{\alpha_0}\right)>0.
\end{equation}
By \eqref{tri-in}, we conclude that, for any $y\in B(x_{\alpha_0}^{k_0},r)$ and $x\in E_{\alpha_0}$,  
$$
d(y,x)\leq A_0\left[d\left(y, x_{\alpha_0}^{k_0}\right)+d\left(x_{\alpha_0}^{k_0},
x\right)\right]\leq 2A_0r,
$$
which further implies that $B(x_{\alpha_0}^{k_0},r)\subset B(x,2A_0r)$.
Combining this and \eqref{eq-doub2}, we find that, 
for any  $x\in E_{\alpha_0}$,
\begin{equation}\label{c-3}
\mu\left(B\left(x_{\alpha_0}^{k_0},r\right)\right)
\leq C_{(\mu)}(2A_0)^\omega\mu(B(x,r)).
\end{equation}
This, together with \eqref{c-1}, further implies that, 
for any  $x\in E_{\alpha_0}$,
$$
|f(x)|> CC_{(\mu)}^{-\theta}(2A_0)^{-\theta\omega} 
\left[\mu\left(B\left(x_{\alpha_0}^{k_0},r\right)\right)\right]^\theta
\|f\|_{\mathrm{lip}_\theta(\mathcal{X})}.
$$
Using this and \eqref{c-2}, we infer that 
\begin{equation}\label{c-4}
\|f\|_{L^\infty(B(x_{\alpha_0}^{k_0},r))}
>CC_{(\mu)}^{-\theta}(2A_0)^{-\theta\omega} 
\left[\mu\left(B\left(x_{\alpha_0}^{k_0},r\right)\right)\right]^\theta
\|f\|_{\mathrm{lip}_\theta(\mathcal{X})}.
\end{equation}
Choose $C:=2C_{(\mu)}^\theta(2A_0)^{\theta\omega}$. 
Then \eqref{c-4} implies that 
$$
\|f\|_{\mathrm{lip}_\theta(\mathcal{X})}
\geq\frac{\|f\|_{L^\infty(B(x_{\alpha_0}^{k_0},r))}}{[\mu
(B(x_{\alpha_0}^{k_0},r))]^\theta}>
2\|f\|_{\mathrm{lip}_\theta(\mathcal{X})}.
$$
This is a contradiction and hence 
$\mu(E)=0$. This show that \eqref{e-p} holds true for almost every 
$x\in\mathcal{X}$ and finishes the prove of \eqref{e-p}. 

To prove \eqref{e-p2}, we consider the following 
two cases on $d(x,y)$. 

Case 1. $d(x,y)<1$. In this case, since $d(x,y)<1$, if follows that there 
exists $N\in \mathbb{N}$ such that 
\begin{equation}\label{e-dxy}
\frac{N+1}{N}d(x,y)<1.
\end{equation}
Let $B_N:=B(x, \frac{N+1}{N}d(x,y))$. Then $x,y\in B_N$. 
By \eqref{e-dxy}, Definition \ref{d-lip}, \eqref{eq-doub}, 
and $(N+1)/N\leq 2$, 
we obtain 
$$ 
|f(x)-f(y)|\leq \|f\|_{\mathrm{lip}_\theta(\mathcal{X})} [\mu(B_N)]^\theta
\lesssim \|f\|_{\mathrm{lip}_\theta(\mathcal{X})} [V(x,y)]^\theta.
$$
This is the desired estimate. 

Case 2. $d(x,y)\geq 1$. In this case, applying \eqref{e-p}, 
Lemma \ref{l-ball}, and \eqref{eq-doub}, we deduce that 
\begin{align*}
|f(x)-f(y)|&\leq |f(x)|+|f(y)|\\
&\lesssim \|f\|_{\mathrm{lip}_\theta(\mathcal{X})}[V_2(x)]^\theta
+\|f\|_{\mathrm{lip}_\theta(\mathcal{X})}[V_2(y)]^\theta\\
&\lesssim \|f\|_{\mathrm{lip}_\theta(\mathcal{X})}[V_2(x)
+V_2(y)+V(x,y)]^\theta\\
&\lesssim \|f\|_{\mathrm{lip}_\theta(\mathcal{X})}[V(x,y)]^\theta.
\end{align*}
This is the desired estimate. Combining the above two cases, 
we finish the proof of \eqref{e-p2} and hence Lemma \ref{l-point}.
\end{proof}

Now, we recall the concepts of
test functions and distributions on $\mathcal{X}$; see, for instance,
\cite{hmy06,hmy08}.
For any $\gamma \in (0,\infty)$, the function $P_\gamma$ 
with \emph{Polynomial decay} is
defined by setting, for any $x,y\in\mathcal{X}$ and
$r\in(0,\infty)$,
\begin{equation}\label{decay}
P_\gamma(x,y;r):=\frac{1}{V_r(x)+V(x,y)}\left[\frac{r}{r
+d(x,y)}\right]^\gamma.
\end{equation}

\begin{definition}[test functions]
Let $x_0\in\mathcal{X}$, $\beta \in (0,1]$,
and $r, \gamma \in (0,\infty)$.
If a measurable function $f$ on $\mathcal{X}$ satisfies that
there exists a positive constant $C$ such that
\begin{enumerate}
\item[{\rm(i)}] for any $x\in\mathcal{X}$,
\begin{equation}\label{t-size}
|f(x)|\leq CP_\gamma(x_0,x;r);
\end{equation}
\item[{\rm(ii)}] for any $x,y \in \mathcal{X}$ satisfying
$d(x, y)\leq(2A_0)^{-1}[r + d(x_0, x)]$,
\begin{equation}\label{t-reg}
|f(x)-f(y)|\leq C\left[\frac{d(x,y)}{r+d(x_0,x)}\right]^\beta
P_\gamma(x_0,x;r),
\end{equation}
\end{enumerate}
where,  for any $x,y\in\mathcal{X}$ and $r\in(0,\infty)$,
$P_\gamma$ is the same as in \eqref{decay},
then $f$ is
called a \emph{test function of type $(x_0, r, \beta, \gamma)$}.

The symbol $\mathcal{G}(x_0, r, \beta,\gamma)$ denotes the collection
 of all test functions of type $(x_0, r, \beta, \gamma)$.
For any  $f\in \mathcal{G}(x_0, r, \beta, \gamma)$,
its \emph{norm $\|f\|_{\mathcal{G}(x_0, r, \beta, \gamma)}$}
in $\mathcal{G}(x_0, r, \beta, \gamma)$ is defined by setting
$$
\|f\|_{\mathcal{G}(x_0, r, \beta, \gamma)}:=\inf\{C\in(0,\infty):\
\text{\eqref{t-size} and \eqref{t-reg} hold}\}.
$$
\end{definition}

Observe that, for any $x_1,x_2\in\mathcal{X}$ and $r_1,r_2 \in(0,\infty)$,
$$\mathcal{G}(x_1,r_1,\beta,\gamma) 
= \mathcal{G}(x_2,r_2,\beta,\gamma)$$
with equivalent norms while the positive equivalence constants
may depend on $x_1$, $x_2$, $r_1$, and $r_2$.
For a fixed point $x_0\in\mathcal{X}$,
the space $\mathcal{G}(x_0, 1, \beta, \gamma)$ is simplified
by $\mathcal{G}(\beta,\gamma)$.
Usually, $\mathcal{G}(\beta,\gamma)$ 
is called the \emph{spaces of test functions} on $\mathcal{X}$.

Fix $\epsilon\in (0, 1]$ and $\beta, \gamma \in (0, \epsilon]$.
The symbol $\mathcal{G}^\epsilon_0(\beta,\gamma)$  
denotes the completion
of the set $\mathcal{G}(\epsilon, \epsilon)$ in 
$\mathcal{G}(\beta,\gamma)$ with
the norm of $\mathcal{G}^\epsilon_0(\beta,\gamma)$ defined
by setting $\|\cdot\|_{\mathcal{G}^\epsilon_0(\beta,\gamma)}:=
\|\cdot\|_{\mathcal{G}(\beta,\gamma)}$. The dual space
$(\mathcal{G}^\epsilon_0(\beta,\gamma))'$ 
is defined to be the
collection of all continuous linear functionals from
$\mathcal{G}^\epsilon_0(\beta,\gamma)$ to $\mathbb{C}$,
equipped with the weak-$\ast$ topology.
Usually, $(\mathcal{G}^\epsilon_0(\beta,\gamma))'$
is called the
\emph{spaces of distributions} on $\mathcal{X}$.

The following proposition indicates that all test functions are 
\emph{pointwise multipliers} on the inhomogeneous Lipschitz space. 

\begin{proposition}\label{p-pm}
Let $\omega$ be as in \eqref{eq-doub}, $\theta\in (0,1/\omega]$, 
$\beta\in [\theta\omega,1]$, and $\gamma\in (0,\infty)$. Then 
there exists a positive constant $C$ such that, for any 
$\psi \in \mathcal{G}(\beta,\gamma)$ and 
$f\in \mathrm{lip}_\theta(\mathcal{X})$, 
$$
\|\psi f\|_{\mathrm{lip}_\theta(\mathcal{X})}\leq C\|\psi\|_{
\mathcal{G}(\beta,\gamma)}\|f\|_{\mathrm{lip}_\theta(\mathcal{X})}.
$$
\end{proposition}

\begin{proof}
Without loss of generality, we assume that 
$\|f\|_{\mathrm{lip}_\theta(\mathcal{X})}=1$. 
Let $B:=B(x_B,r_B)\subset \mathcal{X}$. We consider 
two cases on $r_B$.

Case 1. $r_B\in(1,\infty)$. In this case, by 
$\mathcal{G}(\beta,\gamma)\subset L^\infty(\mathcal{X})$, we have,
for any $x\in \mathcal{X}$,  
$$
|\psi(x) f(x)|\leq \|\psi\|_{L^\infty(\mathcal{X})}|f(x)|\leq 
\|\psi\|_{\mathcal{G}(\beta,\gamma)}|f(x)|,
$$
which further implies that
\begin{equation}\label{e-r1}
\frac{\|\psi f\|_{L^\infty(B)}}{[\mu(B)]^\theta}\leq 
\|\psi\|_{L^\infty(\mathcal{X})}\frac{\|f\|_{L^\infty(B)}}{[\mu(B)]^\theta}
\leq \|\psi\|_{\mathcal{G}(\beta,\gamma)}
\|f\|_{\mathrm{lip}_\theta(\mathcal{X})}.
\end{equation}

Case 2. $r_B\in(0,1]$. In this case, for any $x,y\in B$ with $x\neq y$, 
$$d(x,y)\leq A_0[d(x,x_B)+d(x_B,y)]< 2A_0r_B$$
and hence $V(x,y)\lesssim \mu(B)$. 
Moreover, 
\begin{align*}
|\psi(x) f(x)-\psi(y) f(y)|&=|\psi(x) f(x)-\psi(y) f(x)+\psi(y) f(x)-\psi(y) f(y)|\\
&\leq |\psi(x)-\psi(y)||f(x)|+|\psi(y)||f(x)-f(y)|\\
&=:I_1+I_2.
\end{align*}
For $I_2$, by \eqref{e-p2}, we find that 
\begin{equation}\label{e-i2}
I_2\lesssim \|\psi\|_{\mathcal{G}(\beta,\gamma)}
[V(x,y)]^\theta\lesssim \|\psi\|_{\mathcal{G}(\beta,\gamma)}
[\mu(B)]^\theta.
\end{equation}
To estimate $I_1$, fix $x_0\in\mathcal{X}$. Then, from 
Lemma \ref{l-point}, we deduce that 
\begin{align*}
I_1&\leq |\psi(x)-\psi(y)|[|f(x)-f(x_0)|+|f(x_0)|]\\
&\lesssim |\psi(x)-\psi(y)|
\left\{[V(x,x_0)]^\theta+1\right\}.
\end{align*}
In what follows, we consider two subcases on $d(x,y)$.

Case 2.1. $d(x,y)>(2A_0)^{-1}[d(x,x_0)+1]$. In this case, 
by Lemma \ref{l-ball}(i), 
we conclude that 
\begin{align*}
[V(x,x_0)]^\theta+1&\sim [V(x,x_0)+1]^\theta
\sim [\mu(B(x,d(x,x_0)+1))]^\theta\\
&\lesssim [\mu(B(x,d(x,y)))]^\theta\sim [\mu(B)]^\theta.
\end{align*}
Thus, 
\begin{align}\label{e-i1-1}
I_1&\lesssim [|\psi(x)|+|\psi(y)|]
[\mu(B)]^\theta\lesssim 
\|\psi\|_{\mathcal{G}(\beta,\gamma)}[\mu(B)]^\theta.
\end{align}

Case 2.2. $d(x,y)\leq (2A_0)^{-1} [d(x,x_0)+1]$. In this case, by
Lemmas \ref{l-ball}(i) and \ref{l-ball2}, we obtain 
\begin{align*}
&[V(x,x_0)]^\theta+1\sim [\mu(B(x,d(x,x_0)+1))]^\theta\\
&\quad\lesssim \left[\frac{d(x,x_0)+1}{d(x,y)}\right]^{\theta\omega}
[V(x,y)]^\theta\lesssim \left[\frac{d(x,x_0)+1}{d(x,y)}\right]^{\theta\omega}
[\mu(B)]^\theta.
\end{align*}
Using this, \eqref{t-reg}, and $\beta\in [\theta\omega,1]$, we 
infer that 
\begin{align*}
I_1&\lesssim \|\psi\|_{\mathcal{G}(\beta,\gamma)}
\left[\frac{d(x,y)}{d(x,x_0)+1}\right]^{\beta-\theta\omega}
[\mu(B)]^\theta\notag\\
&\lesssim\|\psi\|_{\mathcal{G}(\beta,\gamma)}
[\mu(B)]^\theta.
\end{align*}
Combining this, \eqref{e-i2}, and \eqref{e-i1-1}, 
we conclude that, 
for any $x,y\in B$, 
$$
\frac{|\psi(x) f(x)-\psi(y) f(y)|}{[\mu(B)]^\theta}
\lesssim\|\psi\|_{\mathcal{G}(\beta,\gamma)},
$$
which, together with \eqref{e-r1}, further implies that 
$$
\|\psi f\|_{\mathrm{lip}_\theta(\mathcal{X})}\leq C\|\psi\|_{
\mathcal{G}(\beta,\gamma)}.
$$
This finishes the proof of Proposition \ref{p-pm}.
\end{proof}

\begin{proposition}\label{p-lipd}
Let $\omega$ be as in \eqref{eq-doub}, $\theta\in (0,\infty)$, 
$\beta\in (0,1]$, and $\gamma\in (\theta\omega,\infty)$. Then 
there exists a positive constant $C$ such that, for any 
$\psi \in \mathcal{G}(\beta,\gamma)$ and 
$f\in \mathrm{lip}_\theta(\mathcal{X})$, 
\begin{equation}\label{e-f-psi}
|\langle f,\psi\rangle|\leq C\|\psi\|_{
\mathcal{G}(\beta,\gamma)}\|f\|_{\mathrm{lip}_\theta(\mathcal{X})}.
\end{equation}
\end{proposition}

\begin{proof}
Without loss of generality, we assume that 
$\|f\|_{\mathrm{lip}_\theta(\mathcal{X})}=1$. 
To show this proposition, fix $x_0\in\mathcal{X}$, and, for any 
$f\in \mathrm{lip}_\theta(\mathcal{X})$
and $\psi \in \mathcal{G}(\beta,\gamma)$, write 
\begin{align*}
&\left|\int_\mathcal{X}f(x)\psi(x)\,d\mu(x)\right|\\
&\quad\leq \int_\mathcal{X}|f(x)-f(x_0)||\psi(x)|\,d\mu(x)
+|f(x_0)|\int_\mathcal{X}|\psi(x)|\,d\mu(x)\\
&\quad:={\rm I}+{\rm II}.
\end{align*}
For ${\rm I}$, from \eqref{e-p2} and \eqref{t-size}, we deduce that
\begin{equation*}
{\rm I}\lesssim \|\psi\|_{\mathcal{G}(\beta,\gamma)}
\int_{\mathcal{X}}[V(x,x_0)]^\theta
P_\gamma(x,x_0;1)\,d\mu(x),
\end{equation*}
where $P_\gamma$ is as in \eqref{decay}.
Notice that, by Lemma \ref{l-ball}(i) and \eqref{eq-doub}, 
$$
V(x,x_0)\sim \mu(B(x_0,d(x,x_0)))\leq \mu(B(x_0,d(x,x_0)+1))
\lesssim [d(x,x_0)+1]^\omega,
$$
which, combined with Lemma \ref{l-ball}(ii) and $\gamma\in 
(\theta\omega,\infty)$, further implies that 
$$
{\rm I}\lesssim \|\psi\|_{\mathcal{G}(\beta,\gamma)}
\int_{\mathcal{X}}
P_{\gamma-\theta\omega}(x,x_0;1)\,d\mu(x)\lesssim 
\|\psi\|_{\mathcal{G}(\beta,\gamma)}.
$$
For {\rm II}, using \eqref{e-p}, \eqref{t-size}, and Lemma \ref{l-ball}(ii), 
we infer that 
$$
{\rm II}\lesssim \|\psi\|_{\mathcal{G}(\beta,\gamma)}
\int_{\mathcal{X}}
P_{\gamma}(x,x_0;1)\,d\mu(x)\lesssim 
\|\psi\|_{\mathcal{G}(\beta,\gamma)}.
$$
This, together with the estimate of ${\rm I}$, 
finishes the proof of 
Proposition \ref{p-lipd}.
\end{proof}

As a direct corollary of Proposition \ref{p-lipd}, we have the following 
conclusion.

\begin{corollary}
Let $\omega$ be the same as in \eqref{eq-doub}, $\theta\in (0,1/\omega)$, 
$\eta\in (\theta\omega,1]$, $\beta\in(0,\eta]$, 
and $\gamma\in (\theta\omega,\eta)$. Then 
$\mathrm{lip}_\theta(\mathcal{X}) \subset  
(\mathcal{G}_0^\eta(\beta,\gamma))'$ continuously. 
\end{corollary}

\section{Proof of Theorem \ref{thm-lip-cs}}\label{proof}

Now, we establish an equivalence characterization 
of $\mathrm{lip}_\theta(\mathcal{X})$ via  
Carleson sequences. Let us recall the wavelet 
systems in \cite[Theorems 6.1 and 7.1 and 
Corollary 10.4]{ah13}. For any $k\in\mathbb{Z}$, 
denote by $V_k\subset L^2(\mathcal{X})$ 
the closed linear span of spline functions in \cite{ah13}. 
Let $s\in(0,1]$ and $\nu\in(0,\infty)$.  
For any $k\in\mathbb{Z}$, the function $E_k$ with \emph{exponential 
decay} is defined by setting, 
for any $x,y\in\mathcal{X}$ and $r\in(0,\infty)$,
\begin{equation}\label{def-e-decay}
E_k(x,y;r):=\exp\left(-\frac{\nu}{r}
\left[\frac{d(x,y)}{\delta^k}\right]^s\right).
\end{equation}

\begin{lemma}\label{l-wave1}
There exist constants $s\in(0,1]$, $\eta\in(0,1)$,
$C,\nu\in(0,\infty)$, and wavelet functions
$\{\phi_\alpha^k:\ k\in\mathbb{Z}, \alpha\in\mathcal{A}_k\}$  satisfying,
for any $k\in\mathbb{Z}$ and $\alpha\in\mathcal{A}_k$,
\begin{enumerate}
\item[{\rm(i)}] for any $x\in \mathcal{X}$,
$$
\left|\phi_\alpha^k(x)\right|\leq
C\left[V_{\delta^k}(x_\alpha^k)\right]^{-1/2}E_k\left(x,x_\alpha^k;1\right);
$$
\item[{\rm(ii)}] for any
$x, x'\in \mathcal{X}$ with $d(x,x')\leq\delta^k$,
$$
\left|\phi_\alpha^k(x)-\phi_\alpha^k(x')\right|
\leq C\left[V_{\delta^k}(x_\alpha^k)\right]^{-1/2}
\left[\frac{d(x,x')}{\delta^k}\right]^\eta
E_k\left(x,x_\alpha^k;1\right),
$$
\end{enumerate}
where $E_k$ is the same as in \eqref{def-e-decay}. 
Moreover, for any $k\in\mathbb{Z}$, the functions 
$\{\phi_\alpha^k\}_{k}$ form
an orthonormal base of $V_k$.
\end{lemma}

\begin{lemma}\label{l-wave2}
There exist constants $s\in(0,1]$, $\eta\in(0,1)$,
$C,\nu\in(0,\infty)$, and wavelet functions
$\{\psi_\beta^{k+1}:\ k\in\mathbb{Z}, \beta\in\mathcal{G}_k\}$  satisfying,
for any $k\in\mathbb{Z}$ and $\beta\in\mathcal{G}_k$,
\begin{enumerate}
\item[{\rm(i)}] for any $x\in \mathcal{X}$,
$$
\left|\psi_\beta^{k+1}(x)\right|\leq
C\left[V_{\delta^k}(x_\beta^{k+1})\right]^{-1/2}E_k\left(x,x_\beta^{k+1};1\right);
$$
\item[{\rm(ii)}] for any
$x, x'\in \mathcal{X}$ with $d(x,x')\leq\delta^k$,
$$
\left|\psi_\beta^{k+1}(x)-\psi_\beta^{k+1}(x')\right|
\leq
C\left[V_{\delta^k}(x_\beta^{k+1})\right]^{-1/2}
\left[\frac{d(x,x')}{\delta^k}\right]^\eta
E_k\left(x,x_\beta^{k+1};1\right);
$$
\item[{\rm(iii)}]
$$
\int_{\mathcal{X}} \psi_\beta^{k+1}(x)\,d\mu(x)=0,
$$
\end{enumerate}
where $E_k$ is the same as in \eqref{def-e-decay}. 
Moreover, the functions $\{\psi_\alpha^k\}_{k,\alpha}$ form
an orthonormal base of $L^2(\mathcal{X})$ and an unconditional
base of $L^p(\mathcal{X})$ for any given $p\in(1,\infty)$.
\end{lemma}

\begin{remark}
\begin{enumerate}
\item[{\rm (i)}] The constant $\eta$ in Lemmas \ref{l-wave1} and \ref{l-wave2} 
comes from the construction of 
random dyadic cubes in \cite{ah13}, which is very important 
because it characterizes the smoothness of the wavelets. 
Moreover, from the construction of $\{\psi_\alpha^k\}_{k,\, \beta}$ 
and $\{\phi_\alpha^k\}_{k,\, \alpha}$ in \cite{ah13}, we 
deduce that, for any  $k_0\in\mathbb{Z}$, 
$\{\phi_\alpha^{k_0}\}_{\alpha\in\mathcal{A}_{k_0}}
\cup\{\psi_\beta^{k}:k\in\mathbb{Z},\ k\geq k_0,\ 
\text{and}\ \beta\in\mathcal{G}_k\}$ form
an orthonormal base of $L^2(\mathcal{X})$. 
Moreover, for any $k,l\in\mathbb{Z}$, 
$\alpha\in\mathcal{A}_{k}$, and 
$\beta\in\mathcal{G}_l$, 
$$
\int_{\mathcal{X}} \psi_\beta^{l+1}(x)\phi_\alpha^{k}\,d\mu(x)=0.
$$
\item[{\rm (ii)}] Using the wavelet systems in 
Lemmas \ref{l-wave1} and \ref{l-wave2}, 
He et al. \cite{hlyy}
introduced a kind of approximations of the identity with
exponential decay (for short, exp-ATI) 
and obtained new Calder\'on reproducing formulae
on $\mathcal{X}$, which proves necessary to
establish various real-variable characterizations of Hardy spaces.
Motivated by this, He et al. \cite{hhllyy} developed a complete
real-variable theory of Hardy spaces on $\mathcal{X}$
including various real-variable equivalent characterizations,
which solves an \emph{open problem} on the radial 
characterization of the Hardy space 
on $\mathcal{X}$ raised in \cite{cw77},
and the boundedness of sublinear operators. 
We refer the readers to \cite{fmy19,hyy25,zhy20} 
for more applications of exp-ATIs.
\item[{\rm (iii)}] The constants $s$ and $\nu$ in 
Lemmas \ref{l-wave1} and  \ref{l-wave2} are the same; 
see \cite[Theorems 6.1 and 7.1 and 
Corollary 10.4]{ah13} for more details. 
Thus, in what follows, we always use 
$s$ and $\nu$ to denote the same constant in 
Lemmas \ref{l-wave1} and  \ref{l-wave2}.
\end{enumerate}
\end{remark}

Now, we establish an 
equivalent characterization of  imhomogenous Lipschitz spaces  
via  Carleson sequences. To this end, 
let 
$$
\mathcal{D}_0:=\bigcup_{k=0}^\infty\left\{Q_\alpha^k:
\alpha\in\mathcal{A}_k\right\}.
$$

\begin{proposition}\label{c-lip}
Let $\omega$ be as in \eqref{eq-doub}, $\eta \in (0,1]$ be as in Lemma \ref{l-wave1},  
and $\theta\in (0,\eta/\omega)$. Then  there exists a positive constant 
$C$ such that, for any $f\in\mathrm{lip}_\theta(\mathcal{X})$, 
\begin{align*}
&{}\sup_{Q\in \mathcal{D}_0}\left\{\frac{1}{[\mu(Q)]^{1+2\theta}}\left[
\sum_{\{\alpha\in\mathcal{A}_0:Q_\alpha^0\subset Q\}}\left|
\left\langle f, \phi_\alpha^0\right\rangle\right|^2\right.\right.\\
&\quad\left.\left.+
\sum_{k=0}^\infty\sum_{\{\beta\in\mathcal{G}_k:
Q_\beta^{k+1}\subset Q\}}\left|\left\langle f, 
\psi_\beta^{k+1}\right\rangle\right|^2\right]\right\}^{\frac{1}{2}}
\leq C\|f\|_{\mathrm{lip}_\theta(\mathcal{X})}. 
\end{align*}
\end{proposition}

To prove Proposition \ref{c-lip}, we need the following 
lemma which contains some useful estimates on the pair 
$\langle f, \phi_\alpha^k\rangle$ and $\langle f, \psi_\beta^{k+1}\rangle$. 

\begin{lemma}\label{lip-wave}
Let $\theta\in(0,\infty)$ and $B:=B(x_B,r_B)\subset \mathcal{X}$. 
Then there exists a positive constant $C$
such that, for any $k\in\mathbb{Z}$, $\alpha\in \mathcal{A}_k$, and 
$f\in\mathrm{lip}_\theta(\mathcal{X})$, 
\begin{enumerate}
\item[{\rm (i)}] in general, 
\begin{align*}
&{}\int_{\mathcal{X}}|f(x)-f_B|
\left|\phi_\alpha^k(x)\right|\,d\mu(x)\\
&\quad\leq C\|f\|_{\mathrm{lip}_\theta(\mathcal{X})}
[\mu(B)]^\theta\sqrt{V_{\delta^k}(x_\alpha^k)}\left[1+\frac{\delta^k
+d(x_B,x_\alpha^k)}{r_B}\right]^{\theta\omega}; 
\end{align*}
\item[{\rm (ii)}] if $r_B\in(1,\infty)$, then 
$$
\int_{\mathcal{X}}|f(x)|\left|\phi_\alpha^k(x)\right|\,d
\mu(x)\leq C\|f\|_{\mathrm{lip}_\theta(\mathcal{X})}
[\mu(B)]^\theta\sqrt{V_{\delta^k}(x_\alpha^k)}\left[1+\frac{\delta^k
+d(x_B,x_\alpha^k)}{r_B}\right]^{\theta\omega}; 
$$
\item[{\rm (iii)}] if $\delta^k\leq r_B$ and $d(x_B,x_\alpha^k)<2\tau r_B$
for some $\tau\in[1,\infty)$, then 
\begin{align*}
&\int_{\{x\in\mathcal{X}:d(x_B,x)>4\tau A_0r_B\}}
|f(x)-f_B|\left|\phi_\alpha^k(x)\right|\,d\mu(x)\\
&\quad\leq C\|f\|_{\mathrm{lip}_\theta(\mathcal{X})}
[\mu(B)]^\theta\sqrt{V_{\delta^k}(x_\alpha^k)}
\exp\left[-\frac{\nu}{3}\left(\frac{\tau r_B}{\delta^k}\right)^s\right]; 
\end{align*}
\item[{\rm (iv)}] if $\delta^k\leq r_B$, $r_B\in(1,\infty)$,
and $d(x_B,x_\alpha^k)<2\tau r_B$
for some $\tau\in[1,\infty)$, then 
\begin{align*}
&\int_{\{x\in\mathcal{X}:d(x_B,x)>4\tau A_0r_B\}}
|f(x)|\left|\phi_\alpha^k(x)\right|\,d\mu(x)\\
&\quad\leq C\|f\|_{\mathrm{lip}_\theta(\mathcal{X})}
[\mu(B)]^\theta\sqrt{V_{\delta^k}(x_\alpha^k)}
\exp\left[-\frac{\nu}{3}\left(\frac{\tau r_B}{\delta^k}\right)^s\right]; 
\end{align*}

\item[{\rm (v)}] items through (i) to (iv) 
still hold true if $\phi_\alpha^k$ and 
$x_\alpha^k$ are replaced, respectively, 
by $\psi_\beta^{k+1}$ and $x_\beta^{k+1}$.
\end{enumerate}
\end{lemma}

\begin{proof}
Without loss of generality, we assume that 
$\|f\|_{\mathrm{lip}_\theta(\mathcal{X})}=1$. 
We first prove (i). By \eqref{e-p2}, we have, for any $x\in \mathcal{X}$, 
\begin{align}\label{f-fb}
|f(x)-f_B|&\leq \frac{1}{\mu(B)}\int_{B}|f(x)-f(y)|\,d\mu(y)
\lesssim\frac{1}{\mu(B)}\int_{B}[V(x,y)]^\theta\,d\mu(y).
\end{align}
From \eqref{tri-in}, we deduce that, for any $y\in B$ 
and $z\in B(x,d(x,y))$, 
\begin{align*}
d(z,x_B)&\leq A_0[d(z,x)+d(x,x_B)]<A_0d(x,y)+A_0d(x,x_B)\\
&\leq A_0^2[d(x,x_B)+d(x_B,y)]+A_0d(x,x_B)\\
&\leq \left(A_0^2+A_0\right)d(x,x_B)+A_0^2r_B,
\end{align*}
Which further implies that $B(x,d(x,y))\subset 
B(x_B, (A_0^2+A_0)d(x,x_B)+A_0^2r_B)$.
By this, we conclude that, for $x\in\mathcal{X}$ and $y\in B$, 
\begin{equation}\label{e-v-b}
V(x,y)\lesssim\left[\frac{r_B+d(x,x_B)}{r_B}\right]^\omega\mu(B).
\end{equation}
By this, \eqref{f-fb}, and Lemma \ref{l-wave1}(i), we obtain 
\begin{align*}
&\int_{\mathcal{X}}|f(x)-f_B|\left|\phi_\alpha^k(x)\right|\,d
\mu(x)\\
&\quad\lesssim [\mu(B)]^\theta
\sqrt{V_{\delta^k}(x_\alpha^k)}\int_{\mathcal{X}}\left[\frac{r_B+
d(x,x_B)}{r_B}\right]^{\theta\omega}
\frac{1}{V_{\delta^k}(x_\alpha^k)}
E_k\left(x,x_\alpha^k;1\right)\,d\mu(x),
\end{align*}
where $E_k$ is as in \eqref{def-e-decay}.
To estimate the above integral, write 
\begin{align*}
&\int_{\mathcal{X}}
\left[\frac{r_B+d(x,x_B)}{r_B}\right]^{\theta\omega}
\frac{1}{V_{\delta^k}(x_\alpha^k)}
E_k\left(x,x_\alpha^k;1\right)\,d\mu(x)\\
&\quad =\int_{B(x_\alpha^k,r_B+d(x_\alpha^k,x_B))}
\left[\frac{r_B+d(x,x_B)}{r_B}\right]^{\theta\omega}
\frac{1}{V_{\delta^k}(x_\alpha^k)}
E_k\left(x,x_\alpha^k;1\right)\,d\mu(x)\\
&\qquad+\int_{\mathcal{X}\setminus B(x_\alpha^k,
r_B+d(x_\alpha^k,x_B))}\cdots\\
&\quad =: {\rm I_1}+{\rm I_2}.
\end{align*}
We first estimate ${\rm I_1}$. Using \eqref{tri-in}, we infer that, 
for any $x\in B(x_\alpha^k,r_B+d(x_\alpha^k,x_B))$, 
\begin{equation}\label{e-i-d}
d(x,x_B)\lesssim d(x,x_\alpha^k)+d(x_\alpha^k,x_B)\lesssim 
r_B+d(x_\alpha^k,x_B).
\end{equation}
From Lemmas \ref{l-ball}(i) and \ref{l-ball2}, we deduce that, 
for any $x\in \mathcal{X}$,
\begin{align}\label{e-i-b}
\frac{1}{V_{\delta^k}(x_\alpha^k)}
&\sim\frac{1}{V_{\delta^k}
(x_\alpha^k)+V(x_\alpha^k,x)}
\frac{V(x_\alpha^k, \delta^k
+d(x_\alpha^k,x))}{V_{\delta^k}(x_\alpha^k)}\notag\\
&\lesssim \frac{1}{V_{\delta^k}(x_\alpha^k)+V(x_\alpha^k,x)}
\left[\frac{\delta^k+d(x_\alpha^k,x)}{\delta^k}\right]^\omega.
\end{align}
Moreover, notice that, for any $x\in \mathcal{X}$ and $\Gamma\in(0,\infty)$, 
\begin{equation}\label{e-e-p}
E_k\left(x,x_\alpha^k;1\right)
\lesssim 
\left[\frac{\delta^k}{\delta^k+d(x_\alpha^k,x)}\right]^{\Gamma}.
\end{equation}
Combining this with $\Gamma:=\omega+1$, \eqref{e-i-d}, 
\eqref{e-i-b}, and Lemma \ref{l-ball}(ii), 
we find that 
\begin{align}\label{e-i}
{\rm I_1}&\lesssim 
\left[\frac{r_B+d(x_\alpha^k,x_B)}{r_B}\right]^{\theta\omega}
\int_{\mathcal{X}}P_1\left(x,x_\alpha^k;\delta^k\right)\,d\mu(x)\notag\\
&\lesssim  \left[\frac{r_B+d(x_\alpha^k,
x_B)}{r_B}\right]^{\theta\omega},
\end{align}
where $P_1$ is as in \eqref{decay} with $\gamma=1$.
To estimate ${\rm I_2}$, by \eqref{tri-in}, we conclude that, 
for any $x\in \mathcal{X}\setminus B(x_\alpha^k,r_B+d(x_\alpha^k,x_B))$,
\begin{equation*}
r_B+d(x,x_B)\lesssim r_B+d(x,x_\alpha^k)+d(x_\alpha^k,x_B)
\lesssim d(x,x_\alpha^k). 
\end{equation*}
This, together with \eqref{e-i-b}, \eqref{e-e-p} with 
$\Gamma:=\theta\omega+\omega+1$, and Lemma \ref{l-ball}(ii), further 
implies that 
\begin{align*}
{\rm I_2}&\lesssim \left(\frac{\delta^k}{r_B}\right)^{\theta\omega}
\int_\mathcal{X} \left[\frac{d(x,x_\alpha^k)}{\delta^k}\right]^{\theta\omega}
\frac{1}{V_{\delta^k}(x_\alpha^k)+V(x_\alpha^k,x)}\\
&\quad\times\left[\frac{\delta^k+d(x_\alpha^k,x)}{\delta^k}\right]^\omega
\left[\frac{\delta^k}{\delta^k+d(x_\alpha^k,x)}\right]^{\theta
\omega+\omega+1}\,d\mu(x)\\
&\lesssim  \left(\frac{\delta^k}{r_B}\right)^{\theta\omega}
\int_{\mathcal{X}}P_1\left(x,x_\alpha^k;\delta^k\right)\,d\mu(x)
\lesssim \left(\frac{\delta^k}{r_B}\right)^{\theta\omega}.
\end{align*}
Combining the estimates of ${\rm I_1}$ and ${\rm I_2}$, 
we obtain 
$$
\int_{\mathcal{X}}|f(x)-f_B|\left|\phi_\alpha^k(x)\right|\,d
\mu(x)\lesssim
[\mu(B)]^\theta\sqrt{V_{\delta^k}(x_\alpha^k)}\left[1+\frac{\delta^k
+d(x_B,x_\alpha^k)}{r_B}\right]^{\theta\omega}, 
$$
which completes the proof of (i). 

Next, we prove (ii). Since $r_B\in(1,\infty)$, by \eqref{e-p} and 
\eqref{l-ball2}, it follows that, for any $x\in\mathcal{X}$,
\begin{align}\label{e-f-f}
|f(x)|\lesssim [\mu(B(x,r_B))]^\theta
\lesssim [\mu(B)]^\theta
\left[\frac{r_B+d(x,x_B)}{r_B}\right]^{\theta\omega}.
\end{align}
Using this, Lemma \ref{l-wave1}(i), and the estimates of 
${\rm I_1}$ and ${\rm I_2}$, we infer that  
\begin{align*}
&\int_{\mathcal{X}}|f(x)|\left|\phi_\alpha^k(x)\right|\,d
\mu(x)\\
&\quad\lesssim[\mu(B)]^\theta
\sqrt{V_{\delta^k}(x_\alpha^k)}\\
&\qquad\times\int_{\mathcal{X}}\left[\frac{r_B
+d(x,x_B)}{r_B}\right]^{\theta\omega}
\frac{1}{V_{\delta^k}(x_\alpha^k)}
E_k\left(x,x_\alpha^k;1\right)\,d\mu(x)\\
&\quad\lesssim[\mu(B)]^\theta
\sqrt{V_{\delta^k}(x_\alpha^k)}\left[1+\frac{\delta^k
+d(x_B,x_\alpha^k)}{r_B}\right]^{\theta\omega}.
\end{align*}
This finishes the proof of (ii). 

Now, we show (iii). 
By \eqref{f-fb}, \eqref{e-v-b}, and Lemma \ref{l-wave1}(i), we find that 
\begin{align}\label{e-out-1}
&\int_{\{x\in\mathcal{X}:d(x_B,x)>4\tau A_0r_B\}}
|f(x)-f_B|\left|\phi_\alpha^k(x)\right|\,d\mu(x)\notag\\
&\quad\lesssim [\mu(B)]^\theta
\sqrt{V_{\delta^k}(x_\alpha^k)}
\int_{\{x\in\mathcal{X}:d(x_B,x)>4\tau A_0r_B\}}
\left[\frac{r_B+d(x,x_B)}{r_B}\right]^{\theta\omega}\notag\\
&\qquad\times
\frac{1}{V_{\delta^k}(x_\alpha^k)}E_k\left(x,x_\alpha^k;1\right)\,d\mu(x).
\end{align}
Notice that, if $\delta^k\leq r_B$ and 
$d(x_B,x_\alpha^k)<2\tau r_B$
for some $\tau\in[1,\infty)$, then, by \eqref{tri-in}, we have, 
for any $x\in \{x\in\mathcal{X}:d(x_B,x)>4\tau A_0r_B\}$, 
\begin{equation}\label{e-t-r}
4\tau A_0r_B<d(x_B,x)\leq A_0\left[d(x_B,x_\alpha^k)+
d(x_\alpha^k,x)\right]\leq A_0\left[2\tau r_B+
d(x_\alpha^k,x)\right],
\end{equation}
which further implies that $2\tau r_B<d(x_\alpha^k,x)$
and hence 
\begin{align}\label{e-e-dec}
\left[\frac{d(x,x_\alpha^k)}{\delta^k}\right]^s 
&> \frac{1}{3}
\left\{\left[\frac{\tau r_B}{\delta^k}\right]^s
+\left[\frac{d(x,x_\alpha^k)}{4\delta^k}\right]^s
+\left[\frac{d(x,x_\alpha^k)}{4\delta^k}\right]^s\right\}\notag\\
&> \frac{1}{3}
\left\{\left[\frac{\tau r_B}{\delta^k}\right]^s
+\left[\frac{d(x,x_\alpha^k)}{4\delta^k}\right]^s
+\left(\frac{\tau}{2}\right)^s\right\}.
\end{align}
Moreover, from $\delta^k\leq r_B$ and \eqref{e-t-r}, we deduce that 
$$
\frac{r_B+d(x,x_B)}{r_B}\lesssim\tau\frac{r_B+d(x,x_\alpha^k)}{r_B}
\lesssim \tau\left[1+\frac{d(x,x_\alpha^k)}{\delta^k}\right]. 
$$
Applying this, \eqref{e-i-b}, and \eqref{e-e-dec}, we infer that,   
for any $x\in \{x\in\mathcal{X}:d(x_B,x)>4\tau A_0r_B\}$, 
\begin{align*}
&\left[\frac{r_B+d(x,x_B)}{r_B}\right]^{\theta\omega}
\frac{1}{V_{\delta^k}(x_\alpha^k)}
E_k\left(x,x_\alpha^k;1\right)\\
&\quad\lesssim \tau^{\theta\omega}\left[1+
\frac{d(x,x_\alpha^k)}{\delta^k}\right]^{\theta\omega}
\frac{1}{V_{\delta^k}(x_\alpha^k)+V(x_\alpha^k,x)}
\left[\frac{\delta^k+d(x_\alpha^k,x)}{\delta^k}\right]^\omega\\
&\qquad\times \exp\left[-\frac{\nu}{3}\left(
\frac{\tau r_B}{\delta^k}\right)^s\right]
E_k\left(x,x_\alpha^k;3\cdot 4^s\right)
\exp\left(-\frac{\nu\tau^s}{3\cdot 2^s}\right)\\
&\quad\lesssim \exp\left[-\frac{\nu}{3}\left(
\frac{\tau r_B}{\delta^k}\right)^s\right] 
P_1\left(x,x_\alpha^k;\delta^k\right).
\end{align*}
This, together with Lemma \ref{l-ball2}, further implies that 
\begin{align}\label{e-out-2}
&\int_{\{x\in\mathcal{X}:d(x_B,x)>4\tau A_0r_B\}}
\left[\frac{r_B+d(x,x_B)}{r_B}\right]^{\theta\omega}
\frac{1}{V_{\delta^k}(x_\alpha^k)}
E_k\left(x,x_\alpha^k;1\right)\,d\mu(x)\notag\\
&\quad\lesssim \exp\left[-\frac{\nu}{3}\left(
\frac{\tau r_B}{\delta^k}\right)^s\right]\int_{\mathcal{X}}
P_1\left(x,x_\alpha^k;\delta^k\right)\,d\mu(x) 
\lesssim \exp\left[-\frac{\nu}{3}\left(
\frac{\tau r_B}{\delta^k}\right)^s\right]. 
\end{align}
Combining this and \eqref{e-out-1}, we finish the proof of (iii). 

We next prove (iv). By $r_B\in(1,\infty)$, \eqref{e-f-f}, 
\eqref{f-fb}, \eqref{e-v-b}, Lemma \ref{l-wave1}(i), and \eqref{e-out-2}, 
we have 
\begin{align*}
&\int_{\{x\in\mathcal{X}:d(x_B,x)>4\tau A_0r_B\}}
|f(x)|\left|\phi_\alpha^k(x)\right|\,d\mu(x)\notag\\
&\quad\lesssim [\mu(B)]^\theta
\sqrt{V_{\delta^k}(x_\alpha^k)}
\int_{\{x\in\mathcal{X}:d(x_B,x)>4\tau A_0r_B\}}
\left[\frac{r_B+d(x,x_B)}{r_B}\right]^{\theta\omega}\notag\\
&\qquad\times
\frac{1}{V_{\delta^k}(x_\alpha^k)}
E_k\left(x,x_\alpha^k;1\right)\,d\mu(x)\\
&\quad\lesssim [\mu(B)]^\theta
\sqrt{V_{\delta^k}(x_\alpha^k)}\exp\left[-\frac{\nu}{3}\left(
\frac{\tau r_B}{\delta^k}\right)^s\right],
\end{align*}
which completes the proof of (iv). 

Finally, in the proofs of (i) through (iv), we only use the size condition of 
$\phi_\alpha^k$, which $\psi_\beta^{k+1}$ also satisfies. Thus, 
repeating the arguments in the proofs of (i) through (iv), we show that   
items through (i) to (iv) 
still hold true if $\phi_\alpha^k$ and 
$x_\alpha^k$ are replaced, respectively, 
by $\psi_\beta^{k+1}$ and $x_\beta^{k+1}$.
This finishes the proof of (v) and hence of Lemma \ref{lip-wave}.
\end{proof}

Now, we show Proposition \ref{c-lip}.

\begin{proof}[Proof of Proposition \ref{c-lip}]
Without loss of generality, we assume that 
$\|f\|_{\mathrm{lip}_\theta(\mathcal{X})}=1$. 
To prove this proposition, 
fix $Q\in \mathcal{D}_0$.
Notice that, for any $Q\in \mathcal{D}_0$, there exist $k_0\in\mathbb{Z}_+$ 
and $\alpha_0\in\mathcal{A}_{k_0}$ such that $Q=Q_{\alpha_0}^{k_0}$.

We first show that 
\begin{equation}\label{c-lip-1}
\left\{\frac{1}{[\mu(Q_{\alpha_0}^{k_0})]^{1+2\theta}}\left[
\sum_{\{\alpha\in\mathcal{A}_0:Q_\alpha^0
\subset Q_{\alpha_0}^{k_0}\}}\left|
\left\langle f, \phi_\alpha^0\right\rangle\right|^2\right]\right\}^{\frac{1}{2}}
\lesssim1. 
\end{equation}
If $\{\alpha\in\mathcal{A}_0:Q_\alpha^0\subset 
Q_{\alpha_0}^{k_0}\}=\emptyset,$ 
then \eqref{c-lip-1} holds true. If
$\{\alpha\in\mathcal{A}_0:Q_\alpha^0
\subset Q_{\alpha_0}^{k_0}\}\neq\emptyset,$
then, by Lemma \ref{2-cube}(ii) and $k_0\in\mathbb{Z}_+$, 
we find that, for any  
$\tilde{\alpha}\in\{\alpha\in\mathcal{A}_0:Q_\alpha^0
\subset Q_{\alpha_0}^{k_0}\}$, either 
$Q_{\alpha_0}^{k_0}\subset Q_{\tilde{\alpha}}^0$ or 
$Q_{\tilde{\alpha}}^0\cap Q_{\alpha_0}^{k_0}=\emptyset$.
Thus, 
\begin{equation}\label{e-q-q}
Q_{\alpha_0}^{k_0}=Q_{\tilde{\alpha}}^0\quad
\text{and}\quad 
x_{\alpha_0}^{k_0}=x_{\tilde{\alpha}}^0.
\end{equation} 
From Lemma \ref{2-cube}(ii) agagin, we deduce 
that $\{\alpha\in\mathcal{A}_0:Q_\alpha^0\subset Q_{\alpha_0}^{k_0}\}$
has only one element $\tilde{\alpha}$.
To estimate $|\langle f, \phi_{\tilde{\alpha}}^0\rangle|$, let 
$B:=B(x_{\tilde{\alpha}}^0,2)$ and write 
\begin{align*}
\left|\left\langle f, \phi_{\tilde{\alpha}}^0\right\rangle\right|&=\left|
\int_{\mathcal{X}}f(x)\phi_{\tilde{\alpha}}^0(x)\,d\mu(x)\right|\\
&\leq
\int_{\mathcal{X}}|f(x)-f_B|\left|\phi_{\tilde{\alpha}}^0(x)\right|\,d\mu(x)
+|f_B|
\int_{\mathcal{X}}\left|\phi_{\tilde{\alpha}}^0(x)\right|\,d\mu(x)\\
&=: {\rm J_1}+{\rm J_2}.
\end{align*}
From Lemmas \ref{lip-wave}(i) and 
\ref{2-cube}(iii) and \eqref{e-q-q}, we deduce that 
\begin{align*}
{\rm J_1}&\lesssim
[\mu(B)]^\theta\sqrt{V_{1}(x_{\tilde{\alpha}}^0)}\left[1+\frac{1
+d(x_{\tilde{\alpha}}^0,x_{\tilde{\alpha}}^0)}{1}\right]^{\theta\omega}
 \lesssim
\left[\mu\left(Q_{\alpha_0}^{k_0}\right)\right]^{\theta+\frac{1}{2}}.
\end{align*}
For ${\rm J_2}$, notice that, $|f_B|\leq \|f\|_{L^\infty(B)}$. 
On the other hand, using Lemmas \ref{l-wave1}(i), \ref{l-ball}, 
\ref{2-cube}(iii), and \ref{l-ball2}, we infer that 
\begin{align*}
\int_{\mathcal{X}}\left|\phi_{\tilde{\alpha}}^0(x)\right|\,d\mu(x)
&\lesssim
\int_{\mathcal{X}}\frac {1}{\sqrt{V_{1}(x_{\tilde{\alpha}}^0)}}
E_0\left(x,x_{\tilde{\alpha}}^0;1\right)\,d\mu(x)\\
&\lesssim \left[V_{1}(x_{\tilde{\alpha}}^0)\right]^{\frac{1}{2}}
\int_{\mathcal{X}}\frac {V_{d(x,x_{\tilde{\alpha}}^0)
+1}(x_{\tilde{\alpha}}^0)}{V_{1}(x_{\tilde{\alpha}}^0)}\\
&\quad\times\frac{1}{V_{1}(x_{\tilde{\alpha}}^0)
+V(x,x_{\tilde{\alpha}}^0)}E_0\left(x,x_{\tilde{\alpha}}^0;1\right)\,d\mu(x)\\
&\lesssim \left[V_{1}(x_{\tilde{\alpha}}^0)\right]^{\frac{1}{2}}
\int_{\mathcal{X}} P_1\left(x,x_{\tilde{\alpha}}^0;1\right)\,d\mu(x)
\lesssim \left[V_{1}(x_{\tilde{\alpha}}^0)\right]^{\frac{1}{2}},
\end{align*}
where $E_0$ is as in \eqref{def-e-decay} with $k=0$ 
and $P_1$ is as in \eqref{decay} with $\gamma=1$. 
This, together with Lemma \ref{2-cube}(iii), Definition \ref{d-lip}, 
\eqref{eq-doub}, and \eqref{e-q-q}, further implies that 
$$
{\rm J_2}\lesssim \|f\|_{L^\infty(B)}
\left[V_{1}(x_{\tilde{\alpha}}^0)\right]^{\frac{1}{2}}
\lesssim
\left[\mu\left(Q_{\tilde{\alpha}}^0\right)\right]^{\theta+\frac{1}{2}}.
$$
Combining the estimates of ${\rm J_1}$ and ${\rm J_2}$, we conclude that 
\begin{align*}
&\left\{\frac{1}{[\mu(Q_{\alpha_0}^{k_0})]^{1+2\theta}}\left[
\sum_{\{\alpha\in\mathcal{A}_0:Q_\alpha^0
\subset Q_{\alpha_0}^{k_0}\}}\left|
\left\langle f, \phi_\alpha^0\right\rangle\right|^2
\right]\right\}^{\frac{1}{2}}\notag\\
&\quad=\left\{\frac{1}{[\mu(Q_{\alpha_0}^{k_0})]^{1+2\theta}}\left|
\left\langle f, \phi_{\tilde{\alpha}}^0\right\rangle\right|^2\right\}^{\frac{1}{2}}
\lesssim 1. 
\end{align*}
This finishes the proof of \eqref{c-lip-1}. 

Next, we show that 
\begin{equation}\label{c-lip-2}
\left\{\frac{1}{[\mu(Q_{\alpha_0}^{k_0})]^{1+2\theta}}\left[
\sum_{k=0}^\infty\sum_{\{\beta\in\mathcal{G}_k:
Q_\beta^{k+1}\subset Q_{\alpha_0}^{k_0}\}}\left|\left\langle f, 
\psi_\beta^{k+1}\right\rangle\right|^2\right]\right\}^{\frac{1}{2}}
\lesssim 1. 
\end{equation}
For this purpose, write 
\begin{align*}
&\left\{\frac{1}{[\mu(Q_{\alpha_0}^{k_0})]^{1+2\theta}}\left[
\sum_{k=0}^\infty\sum_{\{\beta\in\mathcal{G}_k:
Q_\beta^{k+1}\subset Q_{\alpha_0}^{k_0}\}}\left|\left\langle f, 
\psi_\beta^{k+1}\right\rangle\right|^2\right]\right\}^{\frac{1}{2}}\\
&\quad\lesssim\left\{\frac{1}{[\mu(Q_{\alpha_0}^{k_0})]^{1+2\theta}}\left[
\sum_{k=0}^{k_0-1}\sum_{\{\beta\in\mathcal{G}_k:
Q_\beta^{k+1}\subset Q_{\alpha_0}^{k_0}\}}\left|\left\langle f, 
\psi_\beta^{k+1}\right\rangle\right|^2\right]\right\}^{\frac{1}{2}}\\
&\qquad+\left\{\frac{1}{[\mu(Q_{\alpha_0}^{k_0})]^{1+2\theta}}\left[
\sum_{k=k_0}^\infty\sum_{\{\beta\in\mathcal{G}_k:
Q_\beta^{k+1}\subset Q_{\alpha_0}^{k_0}\}}\left|\left\langle f, 
\psi_\beta^{k+1}\right\rangle\right|^2\right]\right\}^{\frac{1}{2}}\\
&\quad=:{\rm J_3}+{\rm J_4}.
\end{align*}
We now estimate ${\rm J_3}$. Let, if $k_0>0$, 
$$
E:=\left\{(k,\beta):\ k\in\{0,\cdots, k_0-1\}, \beta\in\mathcal{G}_k,
Q_\beta^{k+1}\subset Q_{\alpha_0}^{k_0}\right\}
$$
and, if $k_0=0$, $E:=\emptyset$. If $E=\emptyset$, then ${\rm J_3}=0$.
If $E\neq\emptyset$, we claim that $E$ has only one element. 
Indeed, assume $(k_1,\beta_1),(k_2,\beta_2)\in E$. 
From Lemma \ref{2-cube}(ii) and the definition of $E$, 
we deduce that 
\begin{equation}\label{e-q-p}
Q_{\beta_1}^{k_1+1}=Q_{\alpha_0}^{k_0}=
Q_{\beta_2}^{k_2+1}\quad\text{and}\quad x_{\beta_1}^{k_1+1}
=x_{\alpha_0}^{k_0}=x_{\beta_2}^{k_2+1}.
\end{equation}
Notice that $x_{\beta_1}^{k_1+1}\in \mathcal{Y}^{k_1}$ 
and $x_{\beta_2}^{k_2+1}\in \mathcal{Y}^{k_2}$. By the 
definition of $\mathcal{Y}^{k}$,
we find that, if $k_1\neq k_2$, then 
$\mathcal{Y}^{k_1}\cap\mathcal{Y}^{k_2}=\emptyset$ 
and hence $x_{\beta_1}^{k_1+1}
\neq x_{\beta_2}^{k_2+1}$.  
Therefore, $k_1=k_2$ and $\beta_1=\beta_2$, 
which completes the proof of 
the above claim.  Denote the only element in $E$ by 
$(\tilde{k},\tilde{\beta})$. Then, by \eqref{e-q-p}, we have  
$$
{\rm J_3}=\frac{1}{[\mu(Q_{\alpha_0}^{k_0})]^{\frac{1}{2}+\theta}}
\left|\left\langle f, 
\psi_{\tilde{\beta}}^{\tilde{k}+1}\right\rangle\right|
=\frac{1}{[\mu(Q_{\tilde{\beta}}^{\tilde{k}
+1})]^{\frac{1}{2}+\theta}}\left|\left\langle f, 
\psi_{\tilde{\beta}}^{\tilde{k}+1}\right\rangle\right|.
$$
Let $B_1:=(x_{\tilde{\beta}}^{\tilde{k}+1},\delta^{\tilde{k}+1})$. 
Then, using Lemma \ref{l-wave2}(iii), (i) and (v) of 
Lemma \ref{lip-wave}, and \eqref{e-q-p}, we infer that  
\begin{align*}
{\rm J_3}&=\frac{1}{[\mu(Q_{\tilde{\beta}}^{\tilde{k}
+1})]^{\frac{1}{2}+\theta}}\left|\left\langle f-f_{B_1}, 
\psi_{\tilde{\beta}}^{\tilde{k}+1}\right\rangle\right|\\
&\lesssim \frac{1}{[\mu(Q_{\tilde{\beta}}^{\tilde{k}
+1})]^{\frac{1}{2}+\theta}} [\mu(B_1)]^\theta
\sqrt{V_{\delta^{\tilde{k}}}(x_{\tilde{\beta}}^{\tilde{k}
+1})}\\
&\lesssim 1.
\end{align*}

To estimate ${\rm J_4}$, let $B_2:=B(x_{\alpha_0}^{k_0}, 
2C_\#\delta^{k_0})$ with 
$C_\#$ as the same in Lemma \ref{2-cube}(iii).  
By Lemma \ref{l-wave2}(iii), we have 
\begin{align*}
&\sum_{k=k_0}^\infty\sum_{\{\beta\in\mathcal{G}_k:
Q_\beta^{k+1}\subset Q_{\alpha_0}^{k_0}\}}\left|\left\langle f, 
\psi_\beta^{k+1}\right\rangle\right|^2\\
&\lesssim
\sum_{k=k_0}^\infty\sum_{\{\beta\in\mathcal{G}_k:
Q_\beta^{k+1}\subset Q_{\alpha_0}^{k_0}\}}\left|\left\langle 
[f-f_{B_2}]\mathbf{1}_{4A_0 B_2}, 
\psi_\beta^{k+1}\right\rangle\right|^2\\
&\quad+\sum_{k=k_0}^\infty\sum_{\{\beta\in\mathcal{G}_k:
Q_\beta^{k+1}\subset Q_{\alpha_0}^{k_0}\}}
\left|\left\langle [f-f_{B_2}]\mathbf{1}_{\mathcal{X}\setminus(4A_0B_2)}, 
\psi_\beta^{k+1}\right\rangle\right|^2\\
&\quad=:{\rm J_{4,1}}+{\rm J_{4,2}}.
\end{align*}
For ${\rm J_{4,1}}$, notice that, by Definition \ref{d-lip}, 
we have $[f-f_{B_2}]\mathbf{1}_{4A_0 B_2}\in L^2(\mathcal{X})$. 
From this, Lemma~\ref{l-wave2}, and \eqref{e-p2}, we deduce that 
\begin{align*}
{\rm J_{4,1}}&\leq\left\|[f-f_{B_2}]\mathbf{1}_{4A_0 B_2}
\right\|_{L^2(\mathcal{X})}^2\\
&\leq \int_{4A_0 B_2}\left[\frac{1}{\mu(B_2)}\int_{B_2}
|f(x)-f(y)|\,d\mu(y)\right]^2\,d\mu(x)\\
&\lesssim 
\int_{4A_0 B_2}\left[\frac{1}{\mu(B_2)}\int_{B_2}
[V(x,y)]^\theta\,d\mu(y)\right]^2\,d\mu(x)\\
&\lesssim [\mu(B_2)]^{1+2\theta}.
\end{align*}
For ${\rm J_{4,2}}$, since $k\geq k_0$ and 
$Q_\beta^{k+1}\subset Q_{\alpha_0}^{k_0}$, 
it follows that 
$$
\delta^{k+1}<C_\#\delta^{k_0}
\quad\text{and}\quad d(x_\beta^{k+1}, 
x_{\alpha_0}^{k_0})\leq C_\#\delta^{k_0}.
$$
By this, (iii) and (v) of Lemma \ref{lip-wave}, 
and Lemma \ref{2-cube}(ii), we obtain  
\begin{align*}
{\rm J_{4,2}}&\lesssim \sum_{k=k_0}^\infty
\sum_{\{\beta\in\mathcal{G}_k:
Q_\beta^{k+1}\subset Q_{\alpha_0}^{k_0}\}} 
[\mu(B_2)]^{2\theta}V_{\delta^k}(x_\beta^{k+1})
\exp\left[-\frac{2\nu}{3}
\left(\frac{C_\#\delta^{k_0}}{\delta^k}\right)^s\right]\\
&\sim[\mu(B_2)]^{2\theta}\sum_{k=k_0}^\infty
\exp\left[-\frac{2\nu}{3}\left(\frac{C_\#\delta^{k_0}}{\delta^k}\right)^s
\right]\sum_{\{\beta\in\mathcal{G}_k:
Q_\beta^{k+1}\subset Q_{\alpha_0}^{k_0}\}} \mu(Q_\beta^{k+1})\\
&\lesssim [\mu(B_2)]^{2\theta+1},
\end{align*}
which, together with the estimate of ${\rm J_{4,1}}$, 
further implies that 
$$
{\rm J_4}\lesssim 1.
$$
Combining the estimates of ${\rm J_3}$ and ${\rm J_4}$, 
we find that \eqref{c-lip-2} holds true, which completes 
the proof of Proposition \ref{c-lip}. 
\end{proof}

\begin{proposition}\label{lip-c}
Let $\eta\in(0,1]$ as the same in Lemma \ref{l-wave1}, 
$\theta\in(0,\frac{\eta}{\omega})$. If a sequence 
$$
c:=\left\{c_\alpha^0\right\}_{\alpha\in\mathcal{A}_0}\cup
\left\{c_\beta^{k+1}\right\}_{k\in\mathbb{Z}_+, \beta\in\mathcal{G}_k}
\subset \mathbb{C}
$$ 
satisfies that 
\begin{align}\label{lip-c-c}
\|c\|_\ast&:= \sup_{Q\in \mathcal{D}_0}
\left\{\frac{1}{[\mu(Q)]^{1+2\theta}}\left[
\sum_{\{\alpha\in\mathcal{A}_0:Q_\alpha^0\subset Q\}}\left|
c_\alpha^0\right|^2+
\sum_{k=0}^\infty\sum_{\{\beta\in\mathcal{G}_k:
Q_\beta^{k+1}\subset Q\}}\left| 
c_\beta^{k+1}\right|^2\right]\right\}^{\frac{1}{2}}\notag\\
&<\infty,
\end{align}
then 
$$
\sum_{\alpha\in\mathcal{A}_0} c_\alpha^0
\phi_\alpha^0+\sum_{k=0}^\infty\sum_{\beta\in\mathcal{G}_k} 
c_\beta^{k+1}\psi_\beta^{k+1}
$$
converges in $L_{\mathcal{B}}^2(\mathcal{X})$. 
Denote the limit by $f$. Then there exists a positive 
constant $C$, independent of $c$, such that 
$$
\|f\|_{\mathrm{lip}_\theta(\mathcal{X})}\leq C\|c\|_\ast.
$$ 
\end{proposition}

To prove Proposition \ref{lip-c}, we need several lemmas. 
We first recall the following very useful inequality.

\begin{lemma}\label{l-p}
For any $\theta\in (0,1]$ and $\{a_j\}_{j\in\mathbb{N}}
 \subset \mathbb{C}$,
it holds true that
\begin{equation*}
\left(\sum_{j=1}^\infty|a_j|\right)^\theta\leq
\sum_{j=1}^\infty|a_j|^\theta.
\end{equation*}
\end{lemma} 

The following lemma comes from \cite[Lemma 2.21]{lyy},
 whose proof is still valid if $d$ only satisfies \eqref{tri-in}. 
 We omit details here.

\begin{lemma}\label{sum-e}
There exists a positive constant $C$ such that, for any 
$b,c\in(0,\infty)$, $k\in\mathbb{Z}$, and $x\in\mathcal{X}$, 
$$
\sum_{\alpha\in\mathcal{A}_k}\exp\left(-b
\left[\frac{d(x_\alpha^k,x)}{\delta^k}\right]^c\right)\leq C. 
$$
\end{lemma}

Now, we show Proposition \ref{lip-c}.

\begin{proof}[Proof of Proposition \ref{lip-c}]
Let $B:=(x_B,r_B)\subset \mathcal{X}$, we consider two cases on $r_B$. 

Case 1. $r_B\in(1,\infty)$. In this case, write 
\begin{align*}
&\sum_{\alpha\in\mathcal{A}_0} c_\alpha^0\phi_\alpha^0
+\sum_{k=0}^\infty\sum_{\beta\in\mathcal{G}_k} 
c_\beta^{k+1}\psi_\beta^{k+1}\\
&\quad=\sum_{\{\alpha\in\mathcal{A}_0:d(x_\alpha^0,
x_B)<2A_0r_B\}} c_\alpha^0\phi_\alpha^0+ \sum_{\{\alpha\in\mathcal{A}_0:
d(x_\alpha^0,x_B)\geq2A_0r_B\}} c_\alpha^0\phi_\alpha^0\\
&\qquad+\sum_{k=0}^\infty\sum_{\{\beta\in\mathcal{G}_k:
d(x_\beta^{k+1},x_B)<2A_0r_B\}} c_\beta^{k+1}\psi_\beta^{k+1}\\
&\qquad+\sum_{k=0}^\infty\sum_{\{\beta\in\mathcal{G}_k:d(x_\beta^{k+1},
x_B)\geq2A_0r_B\}} c_\beta^{k+1}\psi_\beta^{k+1}\\
&\quad=:F_{1,1}+F_{1,2}+F_{1,3}+F_{1,4}.
\end{align*}
For $F_{1,1}$, by \eqref{tri-in} and Lemma \ref{2-cube}(iii), 
we find that, for any $\alpha\in\mathcal{A}_0$ such that 
$d(x_\alpha^0,x_B)<2A_0r_B$ and any $y\in Q_\alpha^0$,
$$
d(y,x_B)\leq A_0\left[d(y,x_\alpha^0)+d(x_\alpha^0,x_B)\right]
\leq A_0(C_\#+2A_0)r_B,
$$
which further implies that 
$Q_\alpha^0\subset B(x_B,A_0(C_\#+2A_0)r_B)$ and hence 
$$
\left\{\alpha\in\mathcal{A}_0:d(x_\alpha^0,x_B)
<2A_0r_B\right\}\subset 
\left\{\alpha\in\mathcal{A}_0:Q_\alpha^0\subset B(x_B,A_0(C_\#+2A_0)r_B)\right\}.
$$
Using this, Lemma \ref{l-wave1}, \eqref{lip-c-c}, and 
Lemmas \ref{l-p} and \ref{2-cube}(ii), we infer that 
\begin{align*}
\|F_{1,1}\|_{L^2(\mathcal{X})}&\leq 
\left[\sum_{\{\alpha\in\mathcal{A}_0:Q_\alpha^0\subset 
B(x_B,A_0(C_\#+2A_0)r_B)\}} 
\left|c_\alpha^0\right|^2\right]^{\frac{1}{2}}\\
&\leq \|c\|_\ast\left(\sum_{\{\alpha\in\mathcal{A}_0:Q_\alpha^0
\subset B(x_B,A_0(C_\#+2A_0)r_B)\}} \left[\mu(
Q_\alpha^0)\right]^{2\theta+1}\right)^{\frac{1}{2}}\\
&\leq \|c\|_\ast\left[\sum_{\{\alpha\in\mathcal{A}_0:
Q_\alpha^0\subset B(x_B,A_0(C_\#+2A_0)r_B)\}} 
\mu(Q_\alpha^0)\right]^{\theta+\frac{1}{2}}\\
&\lesssim \|c\|_\ast[\mu(B)]^{\theta+\frac{1}{2}}.
\end{align*}
For $F_{1,2}$, by \eqref{lip-c-c} and Lemma \ref{l-wave1}(i), 
we have, for any $x\in B$,  
\begin{align*}
|F_{1,2}(x)|&\leq  \sum_{\{\alpha\in\mathcal{A}_0:
d(x_\alpha^0,x_B)\geq2A_0r_B\}} \left|c_\alpha^0
\right|\left|\phi_\alpha^0(x)\right|\\
&\lesssim \|c\|_\ast \sum_{\{\alpha\in\mathcal{A}_0:
d(x_\alpha^0,x_B)\geq2A_0r_B\}} \left[\mu(Q_\alpha^0
)\right]^\theta E_0\left(x_\alpha^0,x;1\right),
\end{align*}
where $E_0$ is as in \eqref{def-e-decay} with $k=0$. 
Observe that, since $d(x_\alpha^0,x_B)\geq2A_0r_B$, 
it follows that, for any $x\in B$,
\begin{equation}\label{e-a-b}
2A_0r_B\leq d(x_\alpha^0,x_B)\leq A_0\left[d(x_\alpha^0,x)
+d(x,x_B)\right]< A_0d(x_\alpha^0,x)+A_0r_B.
\end{equation}
Thus, 
$d(x,x_B)<r_B<d(x_\alpha^0,x),$
which, together with \eqref{e-a-b}, further implies that 
$$
d(x_\alpha^0,x_B)< 2A_0d(x_\alpha^0,x).
$$
From this and Lemmas \ref{l-ball2} and \ref{sum-e}, 
we deduce that, for any $x\in B$, 
\begin{align*}
|F_{1,2}(x)|&\lesssim \|c\|_\ast \sum_{\{\alpha\in
\mathcal{A}_0:d(x_\alpha^0,x_B)\geq2A_0r_B\}} 
\left[\mu(Q_\alpha^0)\right]^\theta 
E_0\left(x_\alpha^0,x_B;2^sA_0^s\right)\\
&\lesssim \|c\|_\ast [\mu(B)]^\theta \sum_{\{\alpha\in
\mathcal{A}_0:d(x_\alpha^0,x_B)\geq2A_0r_B\}} 
\left[1+d(x_\alpha^0,x_B)\right]^{\theta\omega}
E_0\left(x_\alpha^0,x_B;2^sA_0^s\right)\\
&\lesssim \|c\|_\ast [\mu(B)]^\theta 
\sum_{\{\alpha\in\mathcal{A}_0:d(x_\alpha^0,x_B)\geq2A_0r_B\}} 
E_0\left(x_\alpha^0,x_B;2^{s+1}A_0^s\right)\\
&\lesssim\|c\|_\ast [\mu(B)]^\theta
\end{align*}
and hence
$$
\|F_{1,2}\|_{L^2(B)}\lesssim \|c\|_\ast[\mu(B)]^{\theta+\frac{1}{2}}.
$$

Next, we estimate $F_{1,3}$. Applying \eqref{tri-in} and 
Lemma \ref{2-cube}(iii), we infer that, for any 
$k\in \mathbb{Z}_+$, $\beta\in\mathcal{G}_k$ such 
that $d(x_\beta^{k+1},x_B)<2A_0r_B$ and any 
$y\in Q_\beta^{k+1}$,
$$
d(y,x_B)\leq A_0
\left[d(y,x_\beta^{k+1})+d(x_\beta^{k+1},x_B)\right]
\leq A_0(C_\#+2A_0)r_B,
$$
which further implies that 
$Q_\beta^{k+1}\subset B(x_B,A_0(C_\#+2A_0)r_B)$ 
and hence 
$$
\left\{\beta\in\mathcal{G}_k:d(x_\beta^{k+1},x_B)<2A_0r_B\right\}\subset 
\left\{\beta\in\mathcal{G}_k:Q_\beta^{k+1}\subset 
B(x_B,A_0(C_\#+2A_0)r_B)\right\}.
$$
Using this, Lemma \ref{l-wave2}, \eqref{lip-c-c}, 
and Lemmas \ref{l-p} and \ref{2-cube}(ii), we infer that 
\begin{align*}
\|F_{1,3}\|_{L^2(\mathcal{X})}&\leq \left[\sum_{k=0}^\infty
\sum_{\{\beta\in\mathcal{G}_k:d(x_\beta^{k+1},x_B)
<2A_0r_B\}} \left|c_\beta^{k+1}\right|^2\right]^{\frac{1}{2}}\\
&\leq \left[\sum_{k=0}^\infty\sum_{\{\beta\in\mathcal{G}_k:
Q_\beta^{k+1}\subset B(x_B,A_0(C_\#+2A_0)r_B)\}}
\left|c_\beta^{k+1}\right|^2\right]^{\frac{1}{2}}\\
&\leq \left[\sum_{\{\alpha\in\mathcal{A}_0:Q_\alpha^0
\cap B(x_B,A_0(C_\#+2A_0)r_B)\neq \emptyset\}}
\sum_{k=0}^\infty\sum_{\{\beta\in\mathcal{G}_k:Q_\beta^{k+1}
\subset Q_\alpha^0\}}\left|c_\beta^{k+1}\right|^{2}\right]^{\frac{1}{2}}\\
&\leq \|c\|_\ast\left[\sum_{\{\alpha\in\mathcal{A}_0:
Q_\alpha^0\cap B(x_B,A_0(C_\#+2A_0)r_B)\neq \emptyset\}}
\left[\mu(Q_\alpha^{0})\right]^{2\theta+1}\right]^{\frac{1}{2}}\\
&\leq \|c\|_\ast\left[\sum_{\{\alpha\in\mathcal{A}_0:Q_\alpha^0
\cap B(x_B,A_0(C_\#+2A_0)r_B)\neq \emptyset\}}
\mu(Q_\alpha^{0})\right]^{\frac{1}{2}+\theta}\\
&\lesssim \|c\|_\ast[\mu(B)]^{\theta+\frac{1}{2}}.
\end{align*}

Now, we estimate $F_{1,4}$. By \eqref{lip-c-c} and Lemma \ref{l-wave1}(i), 
we have, for any $x\in B$,  
\begin{align*}
|F_{1,4}(x)|&\leq  \sum_{k=0}^\infty
\sum_{\{\beta\in\mathcal{G}_k:d(x_\beta^{k+1},
x_B)\geq2A_0r_B\}} \left|c_\beta^{k+1}\right|
\left|\psi_\beta^{k+1}(x)\right|\\
&\lesssim \|c\|_\ast \sum_{k=0}^\infty
\sum_{\{\beta\in\mathcal{G}_k:d(x_\beta^{k+1},
x_B)\geq2A_0r_B\}} \left[\mu(Q_\beta^{k+1}
)\right]^\theta E_k\left(x_\beta^{k+
1},x;1\right).
\end{align*}
From $d(x_\beta^{k+1},x_B)\geq2A_0r_B$, 
we deduce that, for any $x\in B$,
\begin{equation*}
2A_0r_B\leq d(x_\beta^{k+1},x_B)\leq A_0\left[d(x_\beta^{k+1},x)
+d(x,x_B)\right]< A_0d(x_\beta^{k+1},x)+A_0r_B,
\end{equation*}
which further implies that
\begin{equation}\label{e-a-b-1}d(x,x_B)<r_B<d(x_\beta^{k+1},x)
\end{equation}
and hence
$$
d(x_\beta^{k+1},x_B)< 2A_0d(x_\beta^{k+1},x).
$$
By this and 
\eqref{e-a-b-1}, we find that, for any $x\in B$, 
$$
d(x_\beta^{k+1},x)>\frac{d(x_\beta^{k+1},x_B)}{4A_0}+\frac{r_B}{2}.
$$
Using this, Lemmas \ref{l-ball2} and \ref{sum-e}, 
we deduce that, for any $x\in B$, 
\begin{align*}
|F_{1,4}(x)|&\lesssim \|c\|_\ast \sum_{k=0}^\infty
\sum_{\{\beta\in\mathcal{G}_k:d(x_\beta^{k+1},
x_B)\geq2A_0r_B\}} \left[\mu(Q_\beta^{k+1}
)\right]^\theta\exp\left(-\nu
\left[\frac{d(x_\beta^{k+1},x_B)}{4A_0\delta^k}
+\frac{r_B}{2\delta^k}\right]^s\right)\\
&\lesssim \|c\|_\ast [\mu(B)]^\theta 
\sum_{k=0}^\infty\sum_{\{\beta\in\mathcal{G}_k:d(x_\beta^{k+1},
x_B)\geq2A_0r_B\}} \left[1+\frac{d(x_\beta^{k+1},x_B)}{r_B}
\right]^{\theta\omega}\\
&\quad\times E_k\left(x_\beta^{k+1},x_B;2^{2s+1}A_0^s\right)
\exp\left[-\frac{\nu}{2^{a+1}}
\left(\frac{r_B}{\delta^k}\right)^s\right]\\
&\lesssim \|c\|_\ast [\mu(B)]^\theta 
\sum_{k=0}^\infty \exp\left[-\frac{\nu}{2^{a+1}}
\left(\frac{1}{\delta^k}\right)^s\right]\\
&\quad\times\sum_{\{\beta\in\mathcal{G}_k:d(x_\beta^{k+1},
x_B)\geq2A_0r_B\}} E_k\left(x_\beta^{k+1},x_B;4^{s+1}A_0^s\right)\\
&\lesssim\|c\|_\ast [\mu(B)]^\theta
\end{align*}
and hence
$$
\|F_{1,4}\|_{L^2(B)}\lesssim \|c\|_\ast[\mu(B)]^{\theta+\frac{1}{2}}.
$$

To summarize the estimates of $F_{1,1}$, $F_{1,2}$, 
$F_{1,3}$, and $F_{1,4}$, we conclude that, 
for almost everywhere $x\in B$, 
$$
\sum_{\alpha\in\mathcal{A}_0} c_\alpha^0
\phi_\alpha^0(x)+\sum_{k=0}^\infty\sum_{\beta\in\mathcal{G}_k} 
c_\beta^{k+1}\psi_\beta^{k+1}(x)
$$
converges. Letting  
\begin{equation}\label{e-d-sum-f}
f:=\sum_{\alpha\in\mathcal{A}_0} c_\alpha^0
\phi_\alpha^0+\sum_{k=0}^\infty\sum_{\beta\in\mathcal{G}_k} 
c_\beta^{k+1}\psi_\beta^{k+1}, 
\end{equation}
then we find that \eqref{e-d-sum-f} converges in 
$L_{\mathcal{B}}^2(\mathcal{X})$ and, 
for any $B\subset \mathcal{X}$ with $r_B\in(1,\infty)$,
\begin{equation}\label{e-f-1}
\left\|f\right\|
_{L^2(B)}
\lesssim \|c\|_\ast[\mu(B)]^{\theta+\frac{1}{2}}.
\end{equation}

Case 2. $r_B\in(0,1]$. In this case, write, for any $x\in B$, 
\begin{align*}
f(x)
&=\sum_{\alpha\in\mathcal{A}_0} c_\alpha^0
\left[\phi_\alpha^0(x)-\phi_\alpha^0(x_B)\right]+ 
\sum_{\alpha\in\mathcal{A}_0} c_\alpha^0\phi_\alpha^0(x_B)\\
&\quad+\sum_{\{k\in\mathbb{Z}_+:\delta^k\leq r_B\}}
\sum_{\{\beta\in\mathcal{G}_k:
d(x_\beta^{k+1},x_B)<2A_0r_B\}} c_\beta^{k+1}\psi_\beta^{k+1}(x)\\
&\quad+\sum_{\{k\in\mathbb{Z}_+:\delta^k\leq r_B\}}
\sum_{\{\beta\in\mathcal{G}_k:d(x_\beta^{k+1},
x_B)\geq2A_0r_B\}} c_\beta^{k+1}\psi_\beta^{k+1}(x)\\
&\quad+\sum_{\{k\in\mathbb{Z}_+:\delta^k> r_B\}}
\sum_{\beta\in\mathcal{G}_k} c_\beta^{k+1}
\left[\psi_\beta^{k+1}(x)-\psi_\beta^{k+1}(x_B)\right]\\
&\quad+\sum_{\{k\in\mathbb{Z}_+:\delta^k>r_B\}}
\sum_{\beta\in\mathcal{G}_k} c_\beta^{k+1}\psi_\beta^{k+1}(x_B)\\
&=:F_{2,1}+F_{2,2}+F_{2,3}+F_{2,4}+F_{2,5}+F_{2,6}.
\end{align*} 
We first estimate $F_{2,1}$. Applying Lemma \ref{l-wave1}(ii), 
\eqref{lip-c-c}, Lemma \ref{l-ball2}, $r_B\in(0,1]$, 
$\theta\in(0,\frac{\eta}{\omega})$, and Lemma \ref{sum-e}, 
we infer that, for any $x\in B$,  
\begin{align*}
|F_{2,1}(x)|&\lesssim \|c\|_\ast\sum_{\alpha\in\mathcal{A}_0}
\left[\mu(Q_\alpha^0
)\right]^\theta[d(x,x_B)]^\eta 
E_0\left(x_\alpha^0,x_B;1\right)\\
&\lesssim \|c\|_\ast[\mu(B)]^\theta\sum_{\alpha\in\mathcal{A}_0} 
\left[\frac{1+d(x_\alpha^0,x_B)}{r_B}\right]^{\theta\omega}
r_B^\eta E_0\left(x_\alpha^0,x_B;1\right)\\
&\lesssim \|c\|_\ast[\mu(B)]^\theta\sum_{\alpha\in\mathcal{A}_0} 
E_0\left(x_\alpha^0,x_B;2\right)
\lesssim \|c\|_\ast[\mu(B)]^\theta, 
\end{align*}
which further implies that 
$$
\|F_{2,1}\|_{L^2(B)}\lesssim \|c\|_\ast[\mu(B)]^{\theta+\frac{1}{2}}.
$$
Now, we estimate $F_{2,2}$. 
By Lemma \ref{l-wave1}(i), \eqref{lip-c-c}, and Lemma \ref{l-ball2}, 
we conclude that, 
\begin{align*}
|F_{2,2}|&\lesssim \sum_{\alpha\in\mathcal{A}_0}
\left[\mu(Q_\alpha^0)\right]^\theta 
E_0\left(x_\alpha^0,x_B;1\right)\\
&\lesssim \sum_{\alpha\in\mathcal{A}_0} \left[\frac{1+
d(x_\alpha^0,x_B)}{r_B}\right]^{\theta\omega}
 E_0\left(x_\alpha^0,x_B;1\right)\\
&\lesssim \sum_{\alpha\in\mathcal{A}_0} 
E_0\left(x_\alpha^0,x_B;2\right)\lesssim 1, 
\end{align*}
where the implicit positive constants may depend on $B$.   
Using an argument similar to that used in the estimate of  
$F_{1,3}$ and $F_{1,4}$, we obtain 
$$
\|F_{2,3}+F_{2,4}\|_{L^2(B)}\lesssim \|c\|_\ast[\mu(B)]^{\theta+\frac{1}{2}}.
$$
For $F_{2,5}$, by an argument similar to that used in 
the estimate of $F_{2,1}$, we find that, for any $x\in B$, 
$$
\sum_{\beta\in\mathcal{G}_k} \left|c_\beta^{k+1}\right|
\left|\psi_\beta^{k+1}(x)-\psi_\beta^{k+1}(x_B)\right|
\lesssim \|c\|_\ast[\mu(B)]^\theta\left(\frac{r_B}{\delta^k}
\right)^{\eta-\theta\omega},
$$
which, together with $\theta\in(0,\frac{\eta}{\omega})$, 
further implies that, 
$$
|F_{2,5}(x)|\lesssim \|c\|_\ast[\mu(B)]^\theta
\sum_{\{k\in\mathbb{Z}_+:\delta^k>r_B\}}
\left(\frac{r_B}{\delta^k}\right)^{\eta-\theta\omega}
\lesssim \|c\|_\ast[\mu(B)]^\theta
$$
and hence 
$$
\|F_{2,5}\|_{L^2(B)}\lesssim \|c\|_\ast[\mu(B)]^{\theta+\frac{1}{2}}.
$$
To estimate $F_{2,6}$, using an argument similar to 
that used in the estimate of $F_{2,2}$, we have 
$$
\sum_{\beta\in\mathcal{G}_k} \left|c_\beta^{k+1}\right|
\left|\psi_\beta^{k+1}(x_B)\right|\lesssim 1,
$$
which, combining the fact that 
$\#\{k\in\mathbb{Z}_+:\delta^k>r_B\}$
is a finite number only depending on $r_B$, further implies that 
\begin{align*}
|F_{2,6}|\lesssim 1, 
\end{align*}
where the implicit positive constants may depend on $B$. 
Let $c_B:=F_{2,2}+F_{2,6}$. From the estimates 
of $F_{2,1}$ through $F_{2,6}$, we deduce that 
$$
\|f-c_B\|_{L^2(B)}\lesssim 
\|c\|_\ast[\mu(B)]^{\theta+\frac{1}{2}},
$$
which, together with the H\"older inequality, further implies that 
$$
\|f-f_B\|_{L^2(B)}\leq \|f-c_B\|_{L^2(B)}+\|c_B-f_B\|_{L^2(B)}\lesssim 
\|c\|_\ast[\mu(B)]^{\theta+\frac{1}{2}}.
$$ 
Applying this, \eqref{e-f-1}, and \cite[Corollary 7.5]{hyy19}, 
we infer that $f\in \mathrm{lip}_\theta(\mathcal{X})$ and 
$$
\|f\|_{\mathrm{lip}_\theta(\mathcal{X})}\leq C\|c\|^\ast.
$$
This finishes the proof of Proposition \ref{lip-c}. 
\end{proof}

Using Propositions \ref{c-lip} and \ref{lip-c}, we have 
the following wavelet reproducing formula of functions in  
$\mathrm{lip}_\theta(\mathcal{X})$. 

\begin{proposition}\label{t-lip-wave}
Let $\eta\in(0,1]$ as the same in Lemma \ref{l-wave1}, 
$\theta\in(0,\frac{\eta}{\omega})$. Then, for any 
$f\in \mathrm{lip}_\theta(\mathcal{X})$, 
$$
f=\sum_{\alpha\in\mathcal{A}_0} 
\left\langle f, \phi_\alpha^0\right\rangle\phi_\alpha^0
+\sum_{k=0}^\infty\sum_{\beta\in\mathcal{G}_k} 
\left\langle f,\psi_\beta^{k+1} \right\rangle \psi_\beta^{k+1} 
$$ 
in $L_{\mathcal{B}}^2(\mathcal{X})$. 
\end{proposition}

To prove Proposition \ref{t-lip-wave}, we need the following lemma. 

\begin{lemma}\label{lip-lip}
Let $\eta\in(0,1]$ be as the same in Lemma \ref{l-wave1}, 
$\theta\in(0,\frac{\eta}{\omega})$, and 
$f\in \mathrm{lip}_\theta(\mathcal{X})$. If, 
for any $\alpha\in\mathcal{A}_0$, $\langle f,
\phi_\alpha^0\rangle=0$, and, for any $k\in\mathbb{Z}_+$ 
and $\beta\in\mathcal{G}_k$, $\langle f,
\psi_\beta^{k+1}\rangle=0$, then, for any $x\in\mathcal{X}$, 
$f(x)=0$. 
\end{lemma}

\begin{proof}
Without loss of generality, we assume that 
$\|f\|_{\mathrm{lip}_\theta(\mathcal{X})}=1$. 
To prove Lemma \ref{lip-lip}, we only need to show that, 
for any $\varepsilon\in(0,\infty)$ and 
$B:=B(x_B,r_B)\subset \mathcal{X}$, 
\begin{equation}\label{e-f-inf}
\|f\|_{L^\infty(B)}<\varepsilon.
\end{equation}
Let $\widetilde{B}:=A_0t^2B$, where $t$ is a large 
number such that $A_0t^2r_B>1$. By the definition 
of $\mathrm{lip}_\theta(\mathcal{X})$, we find that 
$\mathbf{1}_{4\widetilde{B}}f\in L^2(\mathcal{X})$. 
Applying Lemmas \ref{l-wave1} and \ref{l-wave2}, write 
\begin{align*}
\mathbf{1}_Bf&=\mathbf{1}_B\mathbf{1}_{4\widetilde{B}}f\\
&=\mathbf{1}_B\sum_{\alpha\in\mathcal{A}_0}
\left\langle \mathbf{1}_{4\widetilde{B}}f, 
\phi_\alpha^0\right\rangle\phi_\alpha^0+\mathbf{1}_B
\sum_{k=0}^\infty\sum_{\beta\in\mathcal{G}_k}
\left\langle \mathbf{1}_{4\widetilde{B}}f, 
\psi_\beta^{k+1}\right\rangle\psi_\beta^{k+1}\\
&=\mathbf{1}_B\sum_{\{\alpha\in\mathcal{A}_0:
d(x_\alpha^0,x_B)<2A_0t^2r_B\}}
\left\langle \mathbf{1}_{4\widetilde{B}}f, 
\phi_\alpha^0\right\rangle\phi_\alpha^0\\
&\quad+\mathbf{1}_B\sum_{\{\alpha\in\mathcal{A}_0:
d(x_\alpha^0,x_B)\geq2A_0t^2r_B\}}
\left\langle \mathbf{1}_{4\widetilde{B}}f, 
\phi_\alpha^0\right\rangle\phi_\alpha^0\\
&\quad + \mathbf{1}_B
\sum_{k=0}^\infty\sum_{\{\beta\in\mathcal{G}_k:
d(x_\beta^{k+1},x_B)<2A_0t^2r_B\}}\left\langle 
\mathbf{1}_{4\widetilde{B}}f, 
\psi_\beta^{k+1}\right\rangle\psi_\beta^{k+1}\\
&\quad+\mathbf{1}_B
\sum_{k=0}^\infty\sum_{\{\beta\in\mathcal{G}_k:
d(x_\beta^{k+1},x_B)\geq2A_0t^2r_B\}}\left\langle 
\mathbf{1}_{4\widetilde{B}}f, 
\psi_\beta^{k+1}\right\rangle\psi_\beta^{k+1}\\
&=:J_1+J_2+J_3+J_4.
\end{align*}

We first consider $J_1$. By $\langle f,\phi_\alpha^0\rangle=0$ 
and Lemma \ref{lip-wave}(iv), 
we have, for any $\alpha\in\mathcal{A}_0$ 
with $d(x_\alpha^0,x_B)<2A_0t^2r_B$,  
\begin{align}\label{e-j-1}
\left|\left\langle \mathbf{1}_{4\widetilde{B}}f, 
\phi_\alpha^0\right\rangle\right|&=\left|\left\langle 
\mathbf{1}_{\mathcal{X}\setminus(4\widetilde{B})}f, 
\phi_\alpha^0\right\rangle\right|\notag\\
&\leq\int_{\{x\in\mathcal{X}:d(x_B,x)>4 
A_0t^2r_B\}}|f(x)|\left|\phi_\alpha^0(x)\right|\,d
\mu(x)\notag\\
&\lesssim[\mu(\widetilde{B})]^\theta\sqrt{V_{1}(x_\alpha^0)}
\exp\left[-\frac{\nu}{3}\left(A_0t^2 r_B\right)^s\right].
\end{align}
From this and Lemmas \ref{l-wave1}(i) and \ref{sum-e}, 
we deduce that, for any $x\in B$, 
\begin{align*}
|J_1(x)|&\leq\sum_{\{\alpha\in\mathcal{A}_0:d(x_\alpha^0,
x_B)<2A_0t^2r_B\}}
\left|\left\langle \mathbf{1}_{4\widetilde{B}}f, \phi_\alpha^0
\right\rangle\right|\left|\phi_\alpha^0(x)\right|\\
&\lesssim
\left[\mu(\widetilde{B})\right]^\theta\exp\left[-\frac{\nu}{3}
\left(A_0t^2 r_B\right)^s\right]\\
&\quad\times\sum_{\{\alpha\in\mathcal{A}_0:d(x_\alpha^0,
x_B)<2A_0t^2r_B\}}\exp\left(-\nu\left[d(x_\alpha^0,x)\right]^s\right)\\
&\lesssim
\left[\mu(\widetilde{B})\right]^\theta\exp\left[-\frac{\nu}{3}
\left(A_0t^2 r_B\right)^s\right].
\end{align*}

To estimate $J_2$, observe that, for any 
$\alpha\in\mathcal{A}_0$ with $d(x_\alpha^0,x_B)
\geq2A_0t^2r_B$ and any $x\in B$, 
$$
2A_0t^2r_B\leq d(x_\alpha^0,x_B)\leq A_0
\left[d(x_\alpha^0,x)+d(x,x_B)\right]<A_0\left[d(x_\alpha^0,x)+r_B\right]
$$
and hence
\begin{equation}\label{e-tb-d}
t^2r_B<d(x_\alpha^0,x).
\end{equation}
Moreover, 
\begin{align*}
1+d(x_B,x_\alpha^0)&\leq 1+
A_0\left[d(x_B,x)+d(x,x_\alpha^0)\right]\\
&<1+A_0r_B+A_0d(x,x_\alpha^0)\\
&< 2A_0t^2r_B+A_0d(x,x_\alpha^0).
\end{align*}
Using this, \eqref{e-tb-d}, and Lemmas \ref{l-wave1}(i) 
and \ref{lip-wave}(ii), we infer that, for any $\alpha\in\mathcal{A}_0$ 
with $d(x_\alpha^0,x_B)\geq2A_0t^2r_B$ and for any $x\in B$, 
\begin{align}\label{e-j-2}
&\left|\left\langle \mathbf{1}_{4\widetilde{B}}f, 
\phi_\alpha^0\right\rangle\right|\left|\phi_\alpha^0(x)\right|\notag\\
&\quad\lesssim
\left[\mu(\widetilde{B})\right]^\theta\left[1+\frac{1
+d(x_B,x_\alpha^0)}{A_0t^2r_B}\right]^{\theta\omega}
E_0\left(x,x_\alpha^0;1\right)\notag\\
&\quad\lesssim
\left[\mu(\widetilde{B})\right]^\theta\left[1
+d(x,x_\alpha^0)\right]^{\theta\omega}
E_0\left(x,x_\alpha^0;1\right)\notag\\
&\quad\lesssim
\left[\mu(\widetilde{B})\right]^\theta
\exp\left[-\frac{\nu}{4}\left(t^2 r_B\right)^s\right]
E_0\left(x,x_\alpha^0;4\right),
\end{align}
where $E_0$ is as in \eqref{def-e-decay} with $k=0$.
This, together with Lemma \ref{sum-e}, further implies that, 
for any $x\in B$, 
\begin{align*}
|J_2(x)|&\leq\sum_{\{\alpha\in\mathcal{A}_0:
d(x_\alpha^0,x_B)\geq2A_0t^2r_B\}}
\left|\left\langle \mathbf{1}_{4\widetilde{B}}f, 
\phi_\alpha^0\right\rangle\right|\left|\phi_\alpha^0(x)\right|\\
&\lesssim 
\left[\mu(\widetilde{B})\right]^\theta
\exp\left[-\frac{\nu}{4}\left(t^2 r_B\right)^s\right]\\
&\quad\times\sum_{\{\alpha\in\mathcal{A}_0:
d(x_\alpha^0,x_B)\geq2A_0t^2r_B\}}
E_0\left(x,x_\alpha^0;4\right)\\
&\lesssim 
\left[\mu(\widetilde{B})\right]^\theta
\exp\left[-\frac{\nu}{4}\left(t^2 r_B\right)^s\right].
\end{align*}

Next, we estimate $J_3$. Using an argument 
similar to that used in \eqref{e-j-1},
we have, for any $\beta\in\mathcal{G}_k$ with 
$d(x_\beta^{k+1},x_B)<2A_0t^2r_B$,  
\begin{align*}
\left|\left\langle \mathbf{1}_{4\widetilde{B}}f, 
\psi_\beta^{k+1}\right\rangle\right|
\lesssim
\left[\mu(\widetilde{B})\right]^\theta
\sqrt{V_{\delta^k}(x_\beta^{k+1})}
\exp\left[-\frac{\nu}{3}\left(\frac{A_0t^2 r_B}{\delta^k}\right)^s\right].
\end{align*}
By this and Lemmas \ref{l-wave2}(i) and \ref{sum-e}, 
we conclude that, for any $x\in B$,
\begin{align*}
|J_3(x)|&\lesssim\sum_{k=0}^\infty\sum_{\{\beta\in\mathcal{G}_k:
d(x_\beta^{k+1},x_B)<2A_0t^2r_B\}}
\left|\left\langle \mathbf{1}_{4\widetilde{B}}f, 
\psi_\beta^{k+1}\right\rangle\right|\left|\psi_\beta^{k+1}(x)\right|\\
&\lesssim 
\left[\mu(\widetilde{B})\right]^\theta\sum_{k=0}^\infty
\exp\left[-\frac{\nu}{3}\left(\frac{A_0t^2 r_B}{\delta^k}\right)^s\right]\\
&\quad\times\sum_{\{\beta\in\mathcal{G}_k:
d(x_\beta^{k+1},x_B)<2A_0t^2r_B\}}
E_k\left(x,x_\beta^{k+1};1\right)\\
&\lesssim 
\left[\mu(\widetilde{B})\right]^\theta
\exp\left[-\frac{\nu}{6}
\left(A_0t^2 r_B\right)^s\right].
\end{align*}

Finally, we estimate $J_4$. Applying an argument 
similar to that used in \eqref{e-j-2}, we infer that, 
for any $\beta\in\mathcal{G}_k$ with $d(x_\beta^{k+1},
x_B)\geq2A_0t^2r_B$,
\begin{align*}
&\left|\left\langle \mathbf{1}_{4\widetilde{B}}f, 
\psi_\beta^{k+1}\right\rangle\right|\left|\psi_\beta^{k+1}(x)\right|\\
&\quad\lesssim
\left[\mu(\widetilde{B})\right]^\theta\exp\left[-\frac{\nu}{4}
\left(\frac{t^2 r_B}{\delta^k}\right)^s\right]
E_k\left(x,x_\beta^{k+1};4\right),
\end{align*}
which, combined with Lemma \ref{sum-e}, 
further implies that, for any $x\in B$, 
\begin{align*}
|J_4(x)|&\lesssim\sum_{k=0}^\infty
\sum_{\{\beta\in\mathcal{G}_k:
d(x_\beta^{k+1},x_B)\geq2A_0t^2r_B\}}
\left|\left\langle \mathbf{1}_{4\widetilde{B}}f, 
\psi_\beta^{k+1}\right\rangle\right|
\left|\psi_\beta^{k+1}(x)\right|\\
&\lesssim 
\left[\mu(\widetilde{B})\right]^\theta\sum_{k=0}^\infty
\exp\left[-\frac{\nu}{4}\left(\frac{t^2 r_B}{\delta^k}\right)^s\right]\\
&\quad\times\sum_{\{\beta\in\mathcal{G}_k:
d(x_\beta^{k+1},x_B)\geq2A_0t^2r_B\}}
E_k\left(x,x_\beta^{k+1};4\right)\\
&\lesssim 
\left[\mu(\widetilde{B})\right]^\theta\exp\left[-\frac{\nu}{8}
\left(t^2 r_B\right)^s\right].
\end{align*}

To summarize the estimates of $J_1$, $J_2$, $J_3$, 
and $J_4$, we find that 
\begin{align*}
\|\mathbf{1}_Bf\|_{L^\infty(B)}&\leq \|J_1\|_{L^\infty(B)}
+\|J_2\|_{L^\infty(B)}+\|J_3\|_{L^\infty(B)}+\|J_4\|_{L^\infty(B)}\\
&\lesssim 
\left[\mu(\widetilde{B})\right]^\theta\exp\left[-\frac{\nu}{8}
\left(t^2 r_B\right)^s\right]\to 0
\end{align*}
as $t\to\infty$. This finishes the proof of 
\eqref{e-f-inf} and hence Lemma \ref{lip-lip}. 
\end{proof}

Now, we show Proposition \ref{t-lip-wave}.

\begin{proof}[Proof of Proposition \ref{t-lip-wave}]
Without loss of generality, we assume that 
$\|f\|_{\mathrm{lip}_\theta(\mathcal{X})}=1$. 
By Propositions \ref{c-lip} and \ref{lip-c}, it follows that 
$$
\sum_{\alpha\in\mathcal{A}_0} 
\left\langle f, \phi_\alpha^0\right\rangle\phi_\alpha^0
+\sum_{k=0}^\infty\sum_{\beta\in\mathcal{G}_k} 
\left\langle f,\psi_\beta^{k+1} \right\rangle \psi_\beta^{k+1} 
$$ 
converges in $L_{\mathcal{B}}^2(\mathcal{X})$. 
Denote this limit by $\widetilde{f}$. To show Proposition \ref{t-lip-wave}, we only need to 
prove that $f=\widetilde{f}$ pointwise. Using 
Lemma \ref{lip-lip}, we further reduce to show that, 
for any $\alpha_0\in\mathcal{A}_0$, 
\begin{equation}\label{e-ff-0}
\left\langle f-\widetilde{f},\phi_{\alpha_0}^0\right\rangle=0,
\end{equation}
and, for any $k_0\in\mathbb{Z}_+$ and 
$\beta_0\in\mathcal{G}_{k_0}$, 
\begin{equation}\label{e-ff-p}
\left\langle f-\widetilde{f},\psi_{\beta_0}^{k_0+1}\right\rangle=0.
\end{equation}
To this end, let, for any $\alpha\in\mathcal{A}_0$, 
$c_\alpha^0:=\langle f, \phi_\alpha^0\rangle$ and, 
for any $k\in\mathbb{Z}_+$ and 
$\beta\in\mathcal{G}_k$,  $c_\beta^{k+1}
:=\langle f,\psi_\beta^{k+1} \rangle$.

We first show \eqref{e-ff-0}. By Lemma \ref{p-lipd},
we find that, for any $\alpha_0\in\mathcal{A}_0$,   
\begin{align*}
\left\langle \widetilde{f},\phi_{\alpha_0}^0\right\rangle
&=\int_\mathcal{X}\widetilde{f}(x)\phi_{\alpha_0}^0(x)\,d\mu(x)\\
&=\sum_{\tilde{\alpha}\in\mathcal{A}_0} \int_\mathcal{X}
\widetilde{f}(x)\mathbf{1}_{Q_{\tilde{\alpha}}^0}(x)
\phi_{\alpha_0}^0(x)\widetilde{f}(x)
\mathbf{1}_{Q_{\tilde{\alpha}}^0}(x)\,d\mu(x)\\
&=\sum_{\tilde{\alpha}\in\mathcal{A}_0}
\sum_{\alpha\in\mathcal{A}_0} 
\left\langle c_\alpha^0\phi_\alpha^0
\mathbf{1}_{Q_{\tilde{\alpha}}^0}, \phi_{\alpha_0}^0
\mathbf{1}_{Q_{\tilde{\alpha}}^0}\right\rangle\\
&\quad+\sum_{\tilde{\alpha}\in\mathcal{A}_0}
\sum_{\alpha\in\mathcal{A}_0}
\left\langle \sum_{\{(k,\beta): k\in\mathbb{Z}_+,
\beta\in\mathcal{G}_k,Q_{\beta}^{k+1}\subset 
Q_{\alpha}^0\}}c_\beta^{k+1}\psi_\beta^{k+1}
\mathbf{1}_{Q_{\tilde{\alpha}}^0}, \phi_{\alpha_0}^0
\mathbf{1}_{Q_{\tilde{\alpha}}^0}\right\rangle\\
&=: {\rm I_1}+{\rm I_2}.
\end{align*}
To estimate ${\rm I_1}$, by Lemma \ref{l-ball}(ii), 
we have, for any $\alpha\in\mathcal{A}_0$,  
\begin{align*}
\left\|\phi_\alpha^0\mathbf{1}_{Q_{\tilde{\alpha}}^0}
\right\|_{L^2(\mathcal{X})}&\lesssim \left[\int_{Q_{\tilde{\alpha}}^0}
\frac{1}{V_1(x_\alpha^0)}E_0\left(x_\alpha^0,
x;\frac12\right)\,d\mu(x)\right]^{\frac{1}{2}},
\end{align*}
where $E_0$ is as in \eqref{def-e-decay} with $k=0$.
If $d(x_\alpha^0,x_{\tilde{\alpha}}^0)\leq 2A_0C_\#$, then, 
for any $x\in Q_{\tilde{\alpha}}^0$, 
$$
d(x_\alpha^0,x)\leq A_0\left[d(x_\alpha^0,x_{\tilde{\alpha}}^0)
+d(x_{\tilde{\alpha}}^0,x)\right]\lesssim 1. 
$$
From this and Lemma \ref{l-ball}, we deduce that 
$$
\left\|\phi_\alpha^0\mathbf{1}_{Q_{\tilde{\alpha}}^0}\right\|_{L^2
(\mathcal{X})}\lesssim 1
\lesssim E_0\left(x_\alpha^0,x_{\tilde{\alpha}}^0;2\right).
$$
If $d(x_\alpha^0,x_{\tilde{\alpha}}^0)>2A_0 C_\#$, then, 
for any $x\in Q_{\tilde{\alpha}}^0$, 
$$
d(x_\alpha^0,x_{\tilde{\alpha}}^0)\leq A_0\left[d(x_\alpha^0,x)
+d(x,x_{\tilde{\alpha}}^0)\right]\leq A_0d(x_\alpha^0,x)
+\frac{1}{2}d(x_\alpha^0,x_{\tilde{\alpha}}^0)
$$
and hence $\frac{1}{2A_0}d(x_\alpha^0,x_{\tilde{\alpha}}^0)
\leq d(x_\alpha^0,x)$. 
Using this and Lemmas \ref{l-ball2} and \ref{l-ball}, we infer that 
\begin{align*}
\left\|\phi_\alpha^0\mathbf{1}_{Q_{\tilde{\alpha}}^0}
\right\|_{L^2(\mathcal{X})}
&\lesssim \left[\int_{Q_{\tilde{\alpha}}^0}\frac{1}{V_1(x_\alpha^0)}
E_0\left(x_\alpha^0,x;\frac12\right)\,d\mu(x)\right]^{\frac{1}{2}}\\
&\lesssim E_0\left(x_\alpha^0,x_{\tilde{\alpha}}^0;2A_0\right)
\left[\int_{Q_{\tilde{\alpha}}^0}\frac{1}{V_1(x_\alpha^0)}
E_0\left(x_\alpha^0,x;1\right)\,d\mu(x)\right]^{\frac{1}{2}}\\
&\lesssim E_0\left(x_\alpha^0,x_{\tilde{\alpha}}^0;2A_0\right).
\end{align*}
By the H\"older inequality, Proposition \ref{c-lip}, and 
Lemmas \ref{l-ball2} and \ref{sum-e}, we conclude that
\begin{align*}
&\sum_{\tilde{\alpha}\in\mathcal{A}_0}
\sum_{\alpha\in\mathcal{A}_0} 
\left|\left\langle c_\alpha^0\phi_\alpha^0
\mathbf{1}_{Q_{\tilde{\alpha}}^0}, \phi_{\alpha_0}^0
\mathbf{1}_{Q_{\tilde{\alpha}}^0}\right\rangle\right|\\
&\quad\lesssim\sum_{\tilde{\alpha}\in\mathcal{A}_0}
\sum_{\alpha\in\mathcal{A}_0} 
\left[\mu(Q_\alpha^0)\right]^{\theta+\frac{1}{2}}\left\|\phi_\alpha^0
\mathbf{1}_{Q_{\tilde{\alpha}}^0}\right\|_{L^2(\mathcal{X})}
\left\|\phi_{\alpha_0}^0\mathbf{1}_{Q_{\tilde{\alpha}}^0}
\right\|_{L^2(\mathcal{X})}\\
&\quad\lesssim \left[\mu(Q_{\alpha_0}^0)\right]^{\theta+\frac{1}{2}}
\sum_{\tilde{\alpha}\in\mathcal{A}_0}\sum_{\alpha\in\mathcal{A}_0}
\left[\frac{\mu(Q_\alpha^0)}{\mu(
Q_{\alpha_0}^0)}\right]^{\theta+\frac{1}{2}}\\
&\qquad\times E_0\left(x_\alpha^0,
x_{\tilde{\alpha}}^0;2A_0\right)
E_0\left(x_{\alpha_0}^0,x_{\tilde{\alpha}}^0;2A_0\right)\\
&\quad\lesssim 
\left[\mu(Q_{\alpha_0}^0)\right]^{\theta+\frac{1}{2}}.
\end{align*}
This implies that ${\rm I_1}$ converges absolutely. 

Next, we estimate ${\rm I_2}$. To this end, we consider 
two cases on $d(x_{\alpha}^{0},x_{\tilde{\alpha}}^{0})$. 

Case 1. $d(x_{\alpha}^{0},x_{\tilde{\alpha}}^{0})
\leq4A_0^2C_\#$. In this case, by Lemmas \ref{l-wave1} 
and \ref{l-wave2}, Proposition \ref{c-lip}, and Lemma \ref{l-ball}(i), 
we obtain 
\begin{align*}
&\left\|\sum_{\{(k,\beta): k\in\mathbb{Z}_+,\beta\in\mathcal{G}_k,
Q_{\beta}^{k+1}\subset Q_{\alpha}^0\}}c_\beta^{k+1}
\psi_\beta^{k+1}\mathbf{1}_{Q_{\tilde{\alpha}}^0}\right\|_{L^2(\mathcal{X})}\\
&\quad \lesssim \left(\sum_{\{(k,\beta): k\in\mathbb{Z}_+,
\beta\in\mathcal{G}_k,Q_{\beta}^{k+1}\subset Q_{\alpha}^0\}}
\left|c_\beta^{k+1}\right|^2\right)^{\frac{1}{2}}\\
&\quad \lesssim\left[\mu(Q_{\alpha}^0)\right]^{\theta+\frac{1}{2}}
\lesssim \left[\mu(Q_{\tilde{\alpha}}^0)\right]^{\theta+\frac{1}{2}}
E_0\left(x_{\alpha}^{0},x_{\tilde{\alpha}}^{0};1\right).
\end{align*}

Case 2. $d(x_{\alpha}^{0},
x_{\tilde{\alpha}}^{0})>4A_0^2C_\#$.
In this case,  for any $(k,\beta)$ such that  
$Q_{\beta}^{k+1}\subset Q_{\alpha}^0$ and any 
$x\in Q_{\tilde{\alpha}}^{0}$,
\begin{equation*}
d(x_{\alpha}^{0},x_{\tilde{\alpha}}^{0})< A_0d(x_{\alpha}^{0},x)
+A_0C_\#<A_0^2C_\#+A_0^2d(x_{\beta}^{k+1},x)+A_0C_\#,
\end{equation*}
which further implies that
\begin{equation}
\frac{1+d(x_{\alpha}^{0},x_{\tilde{\alpha}}^{0})}{4A_0^2}
<d(x_{\beta}^{k+1},x).
\end{equation}
Applying this, Lemma \ref{l-wave2}(i), Proposition \ref{c-lip}, 
and Lemmas \ref{l-ball2} and \ref{sum-e}, 
we deduce that, for any $x\in Q_{\tilde{\alpha}}^{0}$, 
\begin{align*}
&\left|\sum_{\{(k,\beta): k\in\mathbb{Z}_+,\beta\in
\mathcal{G}_k,Q_{\beta}^{k+1}\subset Q_{\alpha}^0\}}
c_\beta^{k+1}\psi_\beta^{k+1}(x)\right|\\
&\quad\lesssim \sum_{k=0}^\infty\sum_{\{\beta
\in\mathcal{G}_k:Q_{\beta}^{k+1}\subset Q_{\alpha}^0\}} 
\left[\mu(Q_\beta^{k+1})\right]^\theta 
E_k\left(x_{\beta}^{k+1},x;2\right)
\exp\left\{-\frac{\nu}{2}
\left[\frac{1+d(x_{\alpha}^{0},x_{\tilde{\alpha}}^{0})}{4
A_0^2\delta^k}\right]^s\right\}\\
&\quad\lesssim  \left[\mu(Q_{\tilde{\alpha}}^{0})\right]^\theta 
\sum_{k=0}^\infty\left[1+d(x_{\alpha}^{0},x_{\tilde{\alpha}}^{0})
\right]^{\theta\omega}\exp\left\{-\frac{\nu}{2}
\left[\frac{1+d(x_{\alpha}^{0},x_{\tilde{\alpha}}^{0})}{4
A_0^2\delta^k}\right]^s\right\} \\
&\qquad\times\sum_{\{\beta\in\mathcal{G}_k:Q_{\beta}^{k+1}
\subset Q_{\alpha}^0\}}
E_k\left(x_{\beta}^{k+1},x;2\right)\\
&\quad\lesssim \left[\mu(Q_{\tilde{\alpha}}^{0})\right]^\theta 
\exp\left\{-\frac{\nu}{4}
\left[\frac{1+d(x_{\alpha}^{0},x_{\tilde{\alpha}}^{0})}{4
A_0^2}\right]^s\right\}\sum_{k=0}^\infty \delta^{k\theta\omega}\\
&\quad \lesssim \left[\mu(Q_{\tilde{\alpha}}^{0})\right]^\theta 
E_0\left(x_{\alpha}^{0},x_{\tilde{\alpha}}^{0};
4^{s+1}A_0^{2s}\right)
\end{align*}
and hence
\begin{align*}
&\left\|\sum_{\{(k,\beta): k\in\mathbb{Z}_+,\beta\in
\mathcal{G}_k,Q_{\beta}^{k+1}\subset Q_{\alpha}^0\}}
c_\beta^{k+1}\psi_\beta^{k+1}
\mathbf{1}_{Q_{\tilde{\alpha}}^0}\right\|_{L^2(\mathcal{X})}\\
&\quad \lesssim \left[\mu(Q_{\tilde{\alpha}}^0)\right]^{\theta
+\frac{1}{2}}E_0\left(x_{\alpha}^{0},x_{\tilde{\alpha}}^{0};
4^{s+1}A_0^{2s}\right).
\end{align*}
Combining the estimates of Cases 1 and 2, 
the H\"older inequality, Proposition \ref{c-lip}, and 
Lemma \ref{sum-e}, we conclude that
\begin{align*}
&\sum_{\tilde{\alpha}\in\mathcal{A}_0}
\sum_{\alpha\in\mathcal{A}_0}
\left|\left\langle \sum_{\{(k,\beta): k\in\mathbb{Z}_+,
\beta\in\mathcal{G}_k,Q_{\beta}^{k+1}
\subset Q_{\alpha}^0\}}c_\beta^{k+1}\psi_\beta^{k+1}
\mathbf{1}_{Q_{\tilde{\alpha}}^0}, \phi_{\alpha_0}^0
\mathbf{1}_{Q_{\tilde{\alpha}}^0}\right\rangle\right|\\
&\quad\lesssim\sum_{\tilde{\alpha}\in\mathcal{A}_0}
\sum_{\alpha\in\mathcal{A}_0}\left\|
\sum_{\{(k,\beta): k\in\mathbb{Z}_+,\beta\in\mathcal{G}_k,
Q_{\beta}^{k+1}\subset Q_{\alpha}^0\}}c_\beta^{k+1}
\psi_\beta^{k+1}\mathbf{1}_{Q_{\tilde{\alpha}}^0}
\right\|_{L^2(\mathcal{X})}\left\|\phi_{\alpha_0}^0
\mathbf{1}_{Q_{\tilde{\alpha}}^0}\right\|_{L^2(\mathcal{X})}\\
&\quad \lesssim\sum_{\tilde{\alpha}\in\mathcal{A}_0}
\sum_{\alpha\in\mathcal{A}_0}
\left[\mu(Q_{\tilde{\alpha}}^0)\right]^{\theta+\frac{1}{2}}
E_0\left(x_{\alpha}^{0},x_{\tilde{\alpha}}^{0};4^{s+1}A_0^{2s}\right)
E_0\left(x_{\alpha}^{0},x_{\tilde{\alpha}}^{0};2A_0\right)\\
&\quad\lesssim  
\left[\mu(Q_{\alpha_0}^0)\right]^{\theta+\frac{1}{2}}, 
\end{align*}
which further implies that ${\rm I_2}$ converges absolutely. 
Therefore, by the Fubini theorem and the orthogonality 
of wavelets, we have 
\begin{align*}
\left\langle \widetilde{f},\phi_{\alpha_0}^0\right\rangle&=
\sum_{\alpha\in\mathcal{A}_0}\sum_{\tilde{\alpha}\in\mathcal{A}_0} 
\left\langle c_\alpha^0\phi_\alpha^0\mathbf{1}_{Q_{\tilde{\alpha}}^0}, 
\phi_{\alpha_0}^0\mathbf{1}_{Q_{\tilde{\alpha}}^0}\right\rangle\\
&\quad+\sum_{\alpha\in\mathcal{A}_0}\sum_{\tilde{\alpha}
\in\mathcal{A}_0}
\left\langle \sum_{\{(k,\beta): k\in\mathbb{Z}_+,\beta\in
\mathcal{G}_k,Q_{\beta}^{k+1}\subset Q_{\alpha}^0\}}
c_\beta^{k+1}\psi_\beta^{k+1}\mathbf{1}_{Q_{\tilde{\alpha}}^0}, 
\phi_{\alpha_0}^0\mathbf{1}_{Q_{\tilde{\alpha}}^0}\right\rangle\\
&=\sum_{\alpha\in\mathcal{A}_0} c_\alpha^0\left\langle 
\phi_\alpha^0, \phi_{\alpha_0}^0\right\rangle+
\sum_{\alpha\in\mathcal{A}_0}
\left\langle \sum_{\{(k,\beta): k\in\mathbb{Z}_+,
\beta\in\mathcal{G}_k,Q_{\beta}^{k+1}\subset Q_{\alpha}^0\}}
c_\beta^{k+1}\psi_\beta^{k+1}, \phi_{\alpha_0}^0\right\rangle\\
&=c_{\alpha_0}^0=\left\langle f,\phi_{\alpha_0}^0\right\rangle.
\end{align*}
This proves \eqref{e-ff-0}.

Now, we show \eqref{e-ff-p}. 
Using an argument similar to that used in the proof of 
\eqref{e-ff-0}, we infer that, for any $k_0\in\mathbb{Z}_+$ 
and $\beta_0\in\mathcal{G}_{k_0}$, 
\begin{equation*}
\left\langle\sum_{\alpha\in\mathcal{A}_0} 
c_\alpha^0\phi_\alpha^0+\sum_{k=k_0}^\infty\sum_{\beta\in\mathcal{G}_k} 
c_\beta^{k+1} \psi_\beta^{k+1},\psi_{\beta_0}^{k_0+1}
\right\rangle=c_{\beta_0}^{k_0+1}=\left\langle f, 
\psi_{\beta_0}^{k_0+1}\right\rangle.
\end{equation*}
Thus, to show \eqref{e-ff-p}, it suffices to show that 
\begin{equation}\label{e-ff-p-1}
\left\langle\sum_{k=0}^{k_0-1}\sum_{\beta\in\mathcal{G}_k} 
c_\beta^{k+1} \psi_\beta^{k+1},\psi_{\beta_0}^{k_0+1}
\right\rangle=0.
\end{equation}
If $k_0=0$, then \eqref{e-ff-p-1} holds true. If $k_0>0$, 
by Lemma \ref{l-wave2}(i) and \eqref{tri-in}, we obtain, 
for any $k\in\{0,\cdots,k_0-1\}$, $\beta\in\mathcal{G}_k$, 
and $x\in\mathcal{X}$,   
\begin{align*}
&\left|\psi_\beta^{k+1}(x)\psi_{\beta_0}^{k_0+1}(x)\right|\\
&\lesssim \frac{1}{\sqrt{V_{\delta^k}(x_\beta^{k+1})}}
E_k\left(x,x_\beta^{k+1};1\right)
E_{k_0}\left(x,x_{\beta_0}^{k_0+1};1\right)\\
&\lesssim \frac{1}{\sqrt{V_{\delta^k}(x_\beta^{k+1})}}
E_k\left(x,x_\beta^{k+1};2\right)\exp\left(-\frac{\nu}{2^{s+1}}\left[
d(x,x_{\beta_0}^{k_0+1})+d(x,x_\beta^{k+1})\right]^s\right)\\
&\lesssim \frac{1}{\sqrt{V_{\delta^k}(x_\beta^{k+1})}}
E_k\left(x,x_\beta^{k+1};2\right)
E_k\left(x_{\beta_0}^{k_0+1},x_\beta^{k+1};
\frac{2^{s+1}A_0}{\delta^{k_0}}\right),
\end{align*}
where the implicit positive constants may depend on $k_0$ and $\beta_0$.
From this, Proposition \ref{lip-c}, and Lemmas \ref{l-ball2}, 
\ref{l-ball}(ii), and \ref{sum-e}, we deduce that
\begin{align*}
&\sum_{k=0}^{k_0-1}\sum_{\beta\in\mathcal{G}_k} 
\left|c_\beta^{k+1}\right|\int_{\mathcal{X}}\left|\psi_\beta^{k+1}(x)
\psi_{\beta_0}^{k_0+1}(x)\right|\,d\mu(x)\\
&\quad\lesssim \sum_{k=0}^{k_0-1}\sum_{\beta\in\mathcal{G}_k}
\left[\mu(Q_\beta^{k+1})\right]^{\theta+\frac{1}{2}}
E_k\left(x_{\beta_0}^{k_0+1},x_\beta^{k+1};\frac{2^{s+1}A_0}{\delta^{k_0}}\right)\\
&\qquad\times\int_{\mathcal{X}}\frac{1}{V_{\delta^k}(x_\beta^{k+1})}
E_k\left(x,x_\beta^{k+1};2\right)\,d\mu(x)\\
&\quad\lesssim \sum_{k=0}^{k_0-1}
\sum_{\beta\in\mathcal{G}_k}E_k\left(x_{\beta_0}^{k_0+1},x_\beta^{k+1};\frac{2^{s+2}A_0}{\delta^{k_0}}\right)\lesssim 1, 
\end{align*}
where the implicit positive constants depend on $k_0$ and 
$\beta_0$. Applying this, the Fubini theorem, 
and the orthogonality of wavelets, we infer that 
\begin{equation*}
\left\langle\sum_{k=0}^{k_0-1}\sum_{\beta\in\mathcal{G}_k} 
c_\beta^{k+1} \psi_\beta^{k+1},\psi_{\beta_0}^{k_0+1}\right\rangle
=\sum_{k=0}^{k_0-1}\sum_{\beta\in\mathcal{G}_k} 
c_\beta^{k+1}\left\langle\psi_\beta^{k+1},
\psi_{\beta_0}^{k_0+1}\right\rangle=0.
\end{equation*}
This finishes the proof of \eqref{e-ff-p-1} and hence 
that of Proposition \ref{t-lip-wave}. 
\end{proof}

Theorem \ref{thm-lip-cs} follows from Propositions \ref{c-lip}, \ref{lip-c}, and \ref{t-lip-wave} derictly; we omit details here. 
Moreover, using Propositions \ref{c-lip} and \ref{lip-c} and 
the proof of Proposition \ref{t-lip-wave}, 
we also have the following conclusion; we omit details here. 

\begin{corollary}\label{c-f-fn}
Let $\eta\in(0,1]$ as the same in Lemma \ref{l-wave1}, 
$\theta\in(0,\frac{\eta}{\omega})$. For any $n\in\mathbb{N}$
and 
$f\in \mathrm{lip}_\theta(\mathcal{X})$, let 
$$
f_n:=\sum_{\alpha\in\mathcal{A}_0} 
\left\langle f, \phi_\alpha^0\right\rangle\phi_\alpha^0
+\sum_{k=0}^n\sum_{\beta\in\mathcal{G}_k} 
\left\langle f,\psi_\beta^{k+1} \right\rangle \psi_\beta^{k+1}, 
$$ 
Then $f=\lim_{n\to\infty} f_n$ converges in 
$L_{\mathcal{B}}^2(\mathcal{X})$ and there exists a positive 
constant $C$, independent of $f$, such that, 
for any $n\in\mathbb{N}$, 
$$
\|f_n\|_{\mathrm{lip}_\theta(\mathcal{X})}\leq C
\|f\|_{\mathrm{lip}_\theta(\mathcal{X})}.
$$  
\end{corollary}

\section{Applications}\label{app}

In this section, we establish several equivalence characterizations 
of geometric conditions on $\mathcal{X}$. The first one 
is as following related to the lower bound.

\begin{corollary}\label{coro-i}
Let $(\mathcal{X},d,\mu)$ be a space of homogeneous type  
and $\theta\in(0,1)$. 
Then $1\in \mathrm{lip}_\theta(\mathcal{X})$ 
if and only if there exists a positive constant $C$ such that, for 
any $x\in \mathcal{X}$, $\mu(x,1)\geq C$.
\end{corollary}

\begin{proof}
We first show the necessity. Assume that $1\in \mathrm{lip}_\theta(\mathcal{X})$.
Then, by the definition of $\mathrm{lip}_\theta(\mathcal{X})$, we have, for any $x\in \mathcal{X}$, 
$$
\frac{\|1\|_{L^\infty(B(x,2))}}{[\mu(B(x,2))]^\theta}\lesssim 1
$$
and hence $1\lesssim [\mu(B(x,2))]^\theta$,
which further implies that $1\lesssim \mu(B(x,1))$. 

Conversely, suppose $1\lesssim \mu(B(x,1))$.  Let $B:=(x_B,r_B)$. If $r_B\in (0,1]$, then 
$$
\sup_{x,y\in B, x\neq y}\frac{|1-1|}{ [\mu(B)]^\theta}=0;
$$
while, if $r_B\in (1,\infty)$, then 
$$
\frac{\|1\|_{L^\infty(B)}}{ [\mu(B)]^\theta}\lesssim \frac{1}{[\mu(B(x_B,1))]^\theta}\lesssim 1.
$$
This implies that $1\in \mathrm{lip}_\theta(\mathcal{X})$ and finishes the proof of 
Corollary \ref{coro-i}.
\end{proof}

Using Theorem \ref{thm-lip-cs}, 
we have one equivalence characterization of the upper bound.  

\begin{corollary}\label{coro-ii}
Let $(\mathcal{X},d,\mu)$ be a space of homogeneous type 
with upper dimension $\omega$, 
$\eta\in(0,1]$ as the same in Lemma \ref{l-wave1}, 
and $\theta\in(0,\frac{\eta}{\omega})$. Then 
the following statements are equivalent.
\begin{enumerate}
\item[{\rm(i)}] $\mathrm{lip}_\theta(\mathcal{X})\subset L^\infty(\mathcal{X})$;
\item[{\rm(ii)}] there exists a positive constant $C$ such that, for 
any $x\in \mathcal{X}$, $\mu(B(x,1))\leq C$.
\end{enumerate}
\end{corollary}

\begin{proof}
(ii) $\Rightarrow$ (i) follows from \eqref{e-p} directly. 
Thus, to show Corollary \ref{coro-ii}, it suffices to prove 
(i) $\Rightarrow$ (ii). 
By Theorem~\ref{thm-lip-cs}, we find that, 
for any  
$\alpha\in\mathcal{A}_0$, 
\begin{equation}\label{e-w-lip}
\left\|\phi_\alpha^{0}\right\|_{\mathrm{lip}_{\theta}(\mathcal{X})}
\sim \left[\mu(Q_\alpha^{0})\right]^{-\frac12-\theta}.
\end{equation}
On the other hand, by the proof of \cite[Theorem 6.1]{ah13}, 
we obtain, for any $\alpha\in\mathcal{A}_0$, 
\begin{equation*}
\left\|\phi_\alpha^{0}\right\|_{L^\infty(\mathcal{X})}
\sim\left|\phi_\alpha^{0}(x_\alpha^{0})\right|
\sim \left[\mu(Q_\alpha^{0})\right]^{-\frac12},
\end{equation*}
which, together with (i) and \eqref{e-w-lip}, further implies that 
$$
\mu(Q_\alpha^{0})\lesssim 1.
$$
For any $x\in \mathcal{X}$, by \eqref{dis-re}, we find that there
exists $\alpha_0\in \mathcal{A}_0$ such that 
$d(x,x_{\alpha_0}^0)<C_0$.
This, combined Lemma \ref{l-ball}(i), further implies that
$$
\mu(B(x,1))\sim \mu(B(x_{\alpha_0}^0,1))\lesssim 1
$$
and hence finishes the proof of Corollary \ref{coro-ii}.
\end{proof}

At the end of this section, we establish a 
equivalence characterization of the so-called 
\emph{Ahlfors regular space} via 
Theorem \ref{thm-lip-cs}.

\begin{corollary}\label{thm-ags}
Let $(\mathcal{X},d,\mu)$ be a space of homogeneous type 
with upper dimension $\omega$, 
$\eta\in(0,1]$ as the same in Lemma \ref{l-wave1}, 
and $\theta\in(0,\frac{\eta}{\omega})$. 
Then the following statements are equivalent.
\begin{enumerate}
\item[{\rm(i)}] $\mathcal{X}$ is an Ahlfors regular space:  
there exists a constant $C\in[1,\infty)$ such that, 
for any $x\in\mathcal{X}$ and $r\in(0,\infty)$, 
\begin{equation}\label{e-agc}
C^{-1}r^\omega\leq\mu(B(x,r))\leq Cr^\omega.
\end{equation}
\item[{\rm(ii)}]  $\mathrm{lip}_\theta(\mathcal{X})
=C^{\theta\omega}(\mathcal{X})$ with equivalence quasi-norms, 
where 
$$C^{\theta\omega}(\mathcal{X}):=\left\{f\in C(\mathcal{X}):
\|f\|_{C^{\theta\omega}(\mathcal{X})}<\infty\right\}$$
with, for any $f\in C(\mathcal{X})$, 
$$\|f\|_{C^{\theta\omega}(\mathcal{X})}:=\sup_{x,\ y\in\mathcal{X},\ x\neq y}
\frac{|f(x)-f(y)|}{[d(x,y)]^{\theta\omega}}
+\|f\|_{L^\infty(\mathcal{X})}.
$$
\end{enumerate}
\end{corollary}

\begin{proof}
We first show (i)\ $\Rightarrow$\ (ii). 
To this end, suppose that
$f\in \mathrm{lip}_\theta(\mathcal{X})$. 
By Lemma \ref{l-point} and \eqref{e-agc}, we find that, 
for any $x\in\mathcal{X}$,
$$
|f(x)|\lesssim\|f\|_{\mathrm{lip}_\theta(\mathcal{X})},
$$
and, for any $x,y\in\mathcal{X}$ with $x\neq y$, 
$$
|f(x)-f(y)|\lesssim 
\|f\|_{\mathrm{lip}_\theta(\mathcal{X})}[V(x,y)]^\theta
\lesssim \|f\|_{\mathrm{lip}_\theta(\mathcal{X})}[d(x,y)]^{\theta\omega}.
$$
This further implies that $f\in C^{\theta\omega}(\mathcal{X})$ and 
$
\|f\|_{C^{\theta\omega}(\mathcal{X})}
\lesssim \|f\|_{\mathrm{lip}_\theta(\mathcal{X})}.
$
On the contrary, assume that $f\in C^{\theta\omega}(\mathcal{X})$ 
and $B:=(x_B,r_B)$ for some $x_B\in \mathcal{X}$ and $r_B\in(0,\infty)$. 
If $r_B\in(1,\infty)$, then, by Definition \ref{d-lip} 
and \eqref{e-agc}, we have
\begin{equation}\label{e-mbf-i}
\mathcal{M}_\theta^B(f)
=\frac{\|f\|_{L^\infty(B)}}{[\mu(B)]^\theta}
\lesssim \|f\|_{L^\infty(\mathcal{X})}
\leq \|f\|_{C^{\theta\omega}(\mathcal{X})}.
\end{equation}
If $r_B\in(0,1]$, then, for any $x,y\in B$ with $x\neq y$, 
$d(x,y)\leq 2A_0r_B$. From this and \eqref{e-agc}, 
we deduce that, for any $x,y\in B$ with $x\neq y$,
$$
|f(x)-f(y)|\leq \|f\|_{C^{\theta\omega}(\mathcal{X})} 
[d(x,y)]^{\theta\omega}
\lesssim \|f\|_{C^{\theta\omega}(\mathcal{X})} 
r^{\theta\omega}\sim  \|f\|_{C^{\theta\omega}(\mathcal{X})} 
[\mu(B)]^{\theta\omega}
$$
and hence 
$$
\mathcal{M}_\theta^B(f)\lesssim 
\|f\|_{C^{\theta\omega}(\mathcal{X})},
$$
which, combined with \eqref{e-mbf-i}, 
further implies that $f\in \mathrm{lip}_\theta(\mathcal{X})$ and 
$
\|f\|_{\mathrm{lip}_\theta(\mathcal{X})}
\lesssim \|f\|_{C^{\theta\omega}(\mathcal{X})}.
$
This finishes the proof of (i)\ $\Rightarrow$\ (ii). 

Next, we show (ii)\ $\Rightarrow$\ (i). 
By \cite[Theorem 6.15]{whhy} and \cite[Theorem 7.4(i)]{hwyy20}, 
we conclude that, for any $k\in\mathbb{Z}_+$ and 
$\beta\in\mathcal{G}_k$, 
\begin{equation}\label{w-n}
\left\|\psi_\beta^{k+1}\right\|_{\mathrm{lip}_{\theta}(\mathcal{X})}
\sim \left\|\psi_\beta^{k+1}\right\|_{\dot{C}^{\theta\omega}(\mathcal{X})}
\sim\delta^{-k\theta\omega}
\left[\mu(Q_\beta^{k+1})\right]^{-\frac12}.
\end{equation} 
On the other hand, from Theorem \ref{thm-lip-cs}, we infer that, 
for any $k\in\mathbb{Z}_+$ and 
$\beta\in\mathcal{G}_k$, 
\begin{equation*}
\left\|\psi_\beta^{k+1}\right\|_{\mathrm{lip}_{\theta}(\mathcal{X})}
\sim \left[\mu(Q_\beta^{k+1})\right]^{-\frac12-\theta}.
\end{equation*}
This, together with \eqref{w-n}, further implies that, 
for any $k\in\mathbb{Z}_+$ and 
$\beta\in\mathcal{G}_k$,
\begin{equation}\label{e-q-k}
\mu(Q_\beta^{k+1})\sim \delta^{k\omega}.
\end{equation}
Using this and \cite[Corollary 3.4]{hhhlp20}, we deduce that, 
for any $x\in\mathcal{X}$ and $r\in(0,\infty)$, 
\begin{equation}\label{e-m-k}
\mu(B(x,r))\gtrsim r^\omega.
\end{equation} 
Moreover, applying an argument similar to that used in the 
proof of \cite[Lemma 2.9]{wyy} and \eqref{e-q-k}, we have, 
for any $k\in\mathbb{Z}_+$ and $\alpha\in\mathcal{A}_k$, 
$$
\mu(Q_\alpha^k)\lesssim \delta^{k\omega}
$$ 
and hence, for any $\alpha\in \mathcal{A}_0$, 
$$
\mu(B(x_\alpha^0,1))\sim \mu(Q_\alpha^k) \lesssim 1.
$$
By this, Lemma \ref{2-cube}(iii), and 
$\mathcal{X}^k\subset\mathcal{X}^0$ for any 
$k\in \mathbb{Z}\setminus \mathbb{Z}_+$, we obtain, for any 
$k\in \mathbb{Z}\setminus \mathbb{Z}_+$ and 
$\alpha\in\mathcal{A}_k$,
$$
\mu(Q_\alpha^k)\leq \mu(B(x_\alpha^k,C_\#\delta^k))
\lesssim \delta^k \mu(B(x_\alpha^k,1))\lesssim \delta^k,
$$
which, combined with \cite[Lemma 2.9]{wyy}, further implies that 
for any $x\in\mathcal{X}$ and $r\in(0,\infty)$, 
\begin{equation*}
\mu(B(x,r))\lesssim r^\omega.
\end{equation*} 
This, together with \eqref{e-m-k}, finishes the proof 
of (ii)\ $\Rightarrow$\ (i) and hence that of 
Corollary \ref{thm-ags}. 
\end{proof}

\paragraph{Acknowledgments.} The author would like to thank Prof.s D. Yang and W. Yuan for their valuable suggestions
on this article.

\bigskip

\noindent Fan Wang (Corresponding author)

\smallskip

\noindent College of Mathematics and Information
Science, Hebei University, Baoding 071002,
The People's Republic of China

\smallskip

\noindent{\it E-mail:} \texttt{fanwang@hbu.edu.cn}

\end{document}